\documentclass[lettersize,journal]{IEEEtran}
\usepackage{amsmath,amsfonts}
\usepackage{amssymb}
\allowdisplaybreaks[4]
\usepackage{algorithmic}
\usepackage{algorithm}
\usepackage{array}
\usepackage[caption=false,font=normalsize,labelfont=sf,textfont=sf]{subfig}
\usepackage{textcomp}
\usepackage{stfloats}
\usepackage{url}
\usepackage{verbatim}
\usepackage{graphicx}
\usepackage{cite}
\usepackage{bm}
\newtheorem{theorem}{Theorem}
\newtheorem{definition}{Definition}

\newtheorem{assumption}{Assumption}

\newtheorem{lemma}{Lemma}

\newtheorem{remark}{Remark}

\newcommand{\E}{\mathbb{E}}

\begin{document}

\title{Compressed Proximal Federated Learning for Non-Convex Composite Optimization on Heterogeneous Data}

\author{Pu Qiu, Chen Ouyang, Yongyang Xiong, Keyou You, 
Wanquan Liu, 
and Yang Shi,  \IEEEmembership{Fellow, IEEE} 
\thanks{This work was supported in part by the National Natural Science Foundation
of China (62203254). 
\emph{(Corresponding author: Yongyang Xiong.)}}
\thanks{P. Qiu, C. Ouyang, Y. Xiong, and W. Liu are with the School of Intelligent Systems Engineering, Sun Yat-Sen University, Shenzhen 518107, P.R China. E-mail: \texttt{xiongyy25@mail.sysu.edu.cn}}
\thanks{K. You is with the Department of Automation, Beijing National Research Center for Information Science and Technology, Tsinghua University, Beijing 100084, China. E-mail: \texttt{youky@tsinghua.edu.cn}}
\thanks{Y. Shi is with the Department of Mechanical Engineering, University of Victoria, Victoria, BC V8W 2Y2, Canada. E-mail: \texttt{yshi@uvic.ca}}
}

\maketitle

\begin{abstract}
Federated Composite Optimization (FCO) has emerged as a promising framework for training models with structural constraints (e.g., sparsity) in distributed edge networks. However, simultaneously achieving communication efficiency and convergence robustness remains a significant challenge, particularly when dealing with non-smooth regularizers, statistical heterogeneity, and the restrictions of biased compression. To address these issues, we propose \textbf{FedCEF} (Federated Composite Error Feedback), a novel algorithm tailored for non-convex FCO. FedCEF introduces a decoupled proximal update scheme that separates the proximal operator from communication, enabling clients to handle non-smooth terms locally while transmitting compressed information. To mitigate the noise from aggressive quantization and the bias from non-IID data, FedCEF integrates a rigorous error feedback mechanism with control variates. Furthermore, we design a communication-efficient pre-proximal downlink strategy that allows clients to exactly reconstruct global control variables without explicit transmission. We theoretically establish that FedCEF achieves sublinear convergence to a bounded residual error under general non-convexity, which is controllable via the step size and batch size. Extensive experiments on real datasets validate FedCEF maintains competitive model accuracy even under extreme compression ratios (e.g., 1\%), significantly reducing the total communication volume compared to uncompressed baselines.
\end{abstract}

\begin{IEEEkeywords}
Non-convex Optimization, Federated Composite Optimization, Data Heterogeneity, Gradient Compression.
\end{IEEEkeywords}

\section{Introduction}
Federated Learning (FL) has emerged as a promising distributed learning framework widely adopted in applications such as Internet of Things (IoT) and cross-institutional collaborations (e.g., healthcare and finance) \cite{lim2020federatedsurvey}, where data privacy and decentralized data storage are critical concerns. It enables multiple clients to train a model collaboratively under a central server's coordination by sharing only model parameters learned from their local data, thus keeping the raw data private\cite{du2021asynchronous,fan2022improving,huang2024federated}.

A key advancement in making FL practically feasible is the Federated Averaging (FedAvg) algorithm \cite{mcmahan2017communication}. Its core innovation lies in performing multiple local stochastic gradient descent steps on each client device between communication rounds, rather than synchronizing gradients at every step as in traditional parallel SGD. This local update mechanism significantly reduces the frequency of communication, which is crucial for deploying FL in bandwidth-constrained environments.

Despite this efficiency gain, FL still struggles with communication bottlenecks and statistical heterogeneity in real-world applications. Firstly, even with reduced communication frequency, the exchange of full model updates can incur substantial overhead, especially as model sizes grow, making scalability a concern under limited bandwidth\cite{konecny2016federated,kairouz2021advances}. Secondly, and more critically, real-world data distributions are typically non-independent and identically distributed (non-IID) across clients. This statistical heterogeneity leads to client drift, where local models diverge towards their respective local optima, hindering convergence to a high-quality global model\cite{zhao2018federated}. This issue is further exacerbated by the common scenario of partial client participation due to unstable connectivity or selective availability. These practical realities underscore the pressing need for FL algorithms that are not only communication-efficient but also inherently robust to heterogeneous data distributions.

Beyond these foundational challenges of smooth federated optimization, many real-world applications necessitate models with specific structures—such as sparsity in feature selection\cite{tibshirani1996regression} or low-rank representations in matrix completion\cite{candes2012exact}—that are naturally formulated as composite optimization problems. This leads to the FCO problem\cite{yuan2021federated}, where the global objective function comprises a smooth, federated loss function and a non-smooth regularization term. This non-smooth term is critical as it enforces desired structural properties on the learned model, moving beyond simple accuracy maximization to address practical constraints like model size and interpretability. Moreover, the loss functions are generally non-convex in deep learning settings, which emphasizes the importance of non-convex Federated Composite Optimization.

Extending federated learning to non-convex composite problems introduces rigorous theoretical and algorithmic challenges. Firstly, naive aggregation strategies suffer from the ``primal averaging curse'', where averaging locally sparse models at the server destroys the low-dimensional structure promoted by the regularizer \cite{yuan2021federated}. Secondly, the inherent non-linearity of the proximal operator invalidates conventional gradient-based analysis and complicates the handling of client drift under statistical heterogeneity \cite{zhang2024composite}. These obstacles are further compounded when aggressive communication compression is employed. The bias introduced by compression mechanisms (e.g., quantization or sparsification) interacts intricately with the non-smooth landscape and non-IID data distributions, creating a complex error accumulation that threatens algorithm stability and convergence accuracy.

Motivated by the above challenges, this paper aims to design a unified algorithm that can jointly ensure provable communication efficiency, robustness to data heterogeneity, and rigorous handling of non-smooth regularizers in the non-convex FCO setting. 

\subsection{Related work}
\textbf{Federated Learning with Statistical Heterogeneity.} The canonical Federated Averaging (FedAvg) algorithm\cite{mcmahan2017communication}, while effective under ideal homogeneous data distributions, suffers from performance degradation under realistic heterogeneous (non-IID) settings. The core issue, known as client drift, arises as local client updates diverge from the global objective, leading to unstable or erratic convergence behavior\cite{zhao2018federated}. Theoretical analyses often rely on a bounded data heterogeneity assumption to establish convergence guarantees for FedAvg and its variants\cite{li2020convergence}. To actively mitigate client drift while maintaining theoretical guarantee, a line of research introduces correction mechanisms. For instance, SCAFFOLD incorporates control variables to correct local updates using global gradient information under non-convex setting, though it incurs doubled communication overhead\cite{karimireddy2020scaffold}. Other methods, such as FedProx\cite{li2020federated}, FedDyn\cite{acar2021federated}, FedPD\cite{zhang2021fedpd}, and FedDC\cite{gao2022feddc} introduce proximal or regularization terms to align local and global objectives. While effective in reducing bias, these approaches may slow down convergence and often still require bounded heterogeneity for their theoretical analysis.

\textbf{Communication Compression in Federated Learning.} To alleviate the communication bottleneck, model updates are typically compressed (e.g., via quantization  or sparsification) before being exchanged between clients and the server \cite{seide20141bit,alistarh2017qsgd,alistarh2017zipml,lim20193lc,xiong2022quantized,huang2025cedas,li2025improving}. 
These techniques are defined by their underlying operators, which are broadly categorized as either unbiased or biased compressors. Theoretical research has traditionally favored unbiased compressors, as their zero-mean property facilitates convergence analysis \cite{khaled2023unified,li2020unified}. In contrast, biased compressors—which often achieve higher empirical compression ratios and stability—have been historically less understood theoretically, and their direct application can impair convergence\cite{beznosikov2023biased}.

A significant line of work addresses the distortion from biased compression through Error Feedback (EF) mechanisms, which accumulate and compensate for compression errors across rounds\cite{stich2018sparsified}. Advanced variants like EF21\cite{richtarik2021ef21}, which compresses deterministic gradient differences, offer stronger convergence guarantees under relaxed assumptions. However, deploying these techniques in practical FL settings faces significant obstacles\cite{basu2019qsparse}. Existing communication-compressed FL methods either lack robustness under the conditions of heterogeneity\cite{li2021page,tian2026nonconvex} and partial participation, or rely on restrictive compressor assumptions (e.g., bounded gradient norm \cite{haddadpour2021federated}). Notably, recent methods such as SCALLION and SCAFCOM \cite{huang2023stochastic} have made progress in enhancing robustness to client heterogeneity and partial participation, while also relaxing assumptions on the compression operators. However, their extension to Federated Composite Optimization (FCO)—where a non-smooth regularizer introduces additional complexity in the interaction between compression, client drift, and optimization stability—remains largely unaddressed.

\textbf{Federated Composite Optimization.} Extending FL to composite problems involving a non-smooth regularizer presents distinct challenges, and research in this setting remains limited. Direct extensions of FedAvg using subgradient methods typically exhibit slow convergence. More dedicated approaches, such as Federated Mirror Descent, are hindered by the “curse of primal averaging,” which destroys global sparsity\cite{yuan2021federated}. Methods based on dual averaging (e.g., FedDA\cite{yuan2021federated} and Fast-FedDA\cite{bao2022fast}) operate in the dual space but have limited theoretical guarantees: FedDA requires quadratic losses under heterogeneity, while Fast-FedDA relies on bounded heterogeneity and strong convexity. Primal-dual methods like FedDR\cite{trandinh2021feddr} and FedADMM\cite{wang2022fedadmm} can handle non-convex losses and heterogeneity by decoupling the loss and regularizer, but they necessitate solving local subproblems with increasing precision, incurring substantial computational cost. Other recent advances, including the methods by Zhang et al.\cite{zhang2026non} and Fedcanon\cite{zhou2025fedcanon}, address non-convex FCO under data heterogeneity with efficient proximal steps, yet they remain fundamentally communication-intensive, requiring the exchange of full-precision model updates. Therefore, a notable gap exists for a communication-efficient algorithm that can jointly handle non-convex composite objectives, data heterogeneity, partial client participation, and provide rigorous convergence guarantees.

\subsection{Contributions}
Our contributions are as follows:
\begin{enumerate}
\item We introduce a unified algorithm FedCEF that rigorously achieves high communication efficiency in non-convex FCO under statistical heterogeneity. The proposed algorithm simultaneously addresses two critical challenges: ensuring valid and efficient optimization of the non-smooth composite objective, and guaranteeing robust convergence under aggressive, biased communication compression.

\item 
 Notably, we theoretically establish that the proposed algorithm achieves an $\mathcal{O}(1/T)$ sublinear convergence to a neighborhood of a stationary point. Importantly, we prove that the radius of this convergence neighborhood can be explicitly controlled by adjusting the global step size and the mini-batch size. Moreover, our analysis holds under remarkably mild conditions, completely eliminating the need for bounded data heterogeneity and accommodating general biased compressors without restrictive assumptions.
\item We conduct experiments to illustrate the effectiveness of FedCEF and support our theories. Our empirical results show that the proposed methods achieve comparable performance to full-precision FCO methods with substantially reduced communication costs. Furthermore, the results confirm the robustness of our approach under varying degrees of data heterogeneity and compression ratios.
\end{enumerate}
\textit{Notations.} 
In the paper, we symbolize the $p$-dimensional vector space as $\mathbb{R}^p$. The $\ell_1$-norm and $\ell_2$-norm are denoted by $\|\cdot\|_1$ and $\|\cdot\|$, respectively. The differential operator is denoted by $\nabla$, additionally, $\mathbb{N}^+$ denotes the set of all positive integers. The notation $f=\mathcal{O}(h)$ means there exists a positive constant $\upsilon<\infty$ such that $f\leq \upsilon h$. 
By $\mathbb{E}[\cdot]$ we denote mathematical expectation. $h'(z)$ denotes an arbitrary subgradient of a function $h$ at $z$, i.e., $h'(z) \in \partial h(z)$.




\section{Problem Formulation}
 We consider minimizing a global objective function $F(x)$, which comprises a general non-convex and smooth loss function $f(x)$ (expressed in a federated finite-sum form) and a non-smooth regularizer $h(x)$:
\begin{equation}\label{eqfco}
\min_{x \in \mathbb{R}^p} F(x) := f(x) + h(x) = \frac{1}{N} \sum_{i=1}^N f_i(x) + h(x),    
\end{equation}
where $N$ is the total number of clients. The local objective $f_i(x) := \mathbb{E}_{\xi_i \sim \mathcal{D}_i} f(x; \xi_i)$ represents the expected loss of the $i$-th client over its local data distribution $\mathcal{D}_i$. We assume that each client $i$ can compute a stochastic gradient $\nabla f(x; \xi_i)$ by sampling data points $\xi_i$ from $\mathcal{D}_i$. Typical instances of $h(x)$ encompass the $\ell_1$-regularizer, nuclear-norm regularizer (applicable to matrix variables), and similar terms that enforce sparsity or low-rankness. The FCO problem reduces to the smooth federated optimization problem when $h \equiv 0$.

Our analysis relies on the following standard assumptions, common in the study of non-convex and non-smooth optimization\cite{mancino2023proximal}.

\begin{assumption} \label{smooth}
The local function $f_{i}:\mathbb{R}^{p}\mapsto\mathbb{R}$ is non-convex and $L$-smooth, i.e., for any $x,y\in\mathbb{R}^{p}$, there exists $L>0$ such that
\begin{subequations}\label{eq:group}
\begin{align}
&f_{i}(x)-f_{i}(y)-\langle\nabla f_{i}(y),x-y\rangle \leq \frac{L}{2}\|x-y\|^{2}, \label{eq:first} \\
&\|\nabla f_{i}(x)-\nabla f_{i}(y)\| \leq L\|x-y\|. \label{eq:second}
\end{align}
\end{subequations}
\end{assumption}

\begin{assumption} \label{assump:subgradient}
The regularizer $h: \mathbb{R}^p \to \mathbb{R} \cup \{+\infty\}$ is proper closed convex, but not necessarily smooth. Additionally, there exists a constant $B_{h} > 0$ such that for any subgradient $h'(z) \in \partial h(z)$, its norm is uniformly bounded: $\|h^{\prime}(z)\|^{2} \leq B_{h}$.
\end{assumption}

This assumption on the subgradient boundedness of the regularizer $h$ is standard in the analysis of non-smooth problems. It is satisfied by common regularizers used for structural learning, such as the $\ell_1$-norm, $\ell_\infty$-norm, and the nuclear norm.

We make the following assumption for the stochastic gradients:
\begin{assumption}[Stochastic Gradient Properties] \label{assump:stoch_grad}
For each client $i$, the stochastic gradient estimator
\begin{equation*}
g_i(x_i) := \frac{1}{B}\sum_{b=1}^{B} \nabla f_i(x_i; \xi_{i,b})
\end{equation*}
constructed from a mini-batch of $B$ i.i.d. samples $\{\xi_{i,b}\}_{b=1}^B \sim \mathcal{D}_i$ satisfies:
\begin{subequations}\label{eq:stoch_properties}
\begin{align}
&\mathbb{E}[g_{i}(x_{i})] = \nabla f_{i}(x_{i}) \quad  \label{eq:unbiased} \\
&\mathbb{E}\|g_{i}(x_{i})-\nabla f_{i}(x_{i})\|^{2} \leqslant \sigma^{2}/B , \label{eq:variance}
\end{align}
\end{subequations}
uniformly for all $x_i \in \mathbb{R}^p$ and some $\sigma > 0$.
\end{assumption}

 In federated learning, clients commonly employ mini-batch stochastic gradients as unbiased, variance-bounded estimations for the true gradients $\nabla f_i(x_i)$. This is a practical necessity in large-scale scenarios where computing exact gradients is computationally prohibitive. Noted that we remove the bounded data heterogeneity assumption as in \cite{tian2026nonconvex}.



We incorporate a communication compression mechanism to extend FCO algorithm to the communication-constrained setting. We specifically focus on a broad class of biased compressors because they typically achieve significantly higher compression ratios than their unbiased counterparts \cite{zhang2023innovation}.
We adopt the widely used contractive compressibility property for the definition of our biased compressors.

\begin{definition}[Contractive Compressor]\label{def:q_contractive_compressor}
A (possibly) randomized mapping $\mathcal{C}:\mathbb{R}^{p}\to\mathbb{R}^{p}$ is defined as a $q^2$-contractive compressor if there exists a compression factor $q \in [0, 1)$ such that for any input vector $x\in\mathbb{R}^{p}$,
\begin{equation}\label{eq:compression_condition_q}
\mathbb{E}\|\mathcal{C}(x)-x\|^{2}\leq q^{2}\|x\|^{2}.
\end{equation}
\end{definition}

Here, $q^2$ quantifies the compression error. For example, the Top-$k$ operator satisfies this with $q^2 = 1 - k/p$. Unlike unbiased compressors, contractive compressors introduce a bias $\mathbb{E}[\mathcal{C}(x)] \neq x$, which our algorithm is specifically designed to correct.

\begin{remark}[Comparison with Existing Assumptions]
\label{rmk:comparison}
It is important to note that our theoretical analysis is built upon a weaker compressor assumption compared to many existing works. Specifically, our framework does not require the restrictive \textit{Bounded Gradient Norm} assumption \cite{basu2019qsparse}:
\[
\max_{1\leq i\leq N}\|\nabla f_{i}(x)\|\leq G,
\]
nor do we rely on the \textit{Averaged Contraction} assumption often required in error-feedback analysis \cite{ li2023federated}:
\begin{equation*}
\mathbb{E}\left\|\frac{1}{N}\sum_{i=1}^{N}\mathcal{C}_{i}(x_{i}) - \frac{1}{N}\sum_{i=1}^{N}x_{i}\right\|^{2}\leq q_{A}^{2} \left\|\frac{1}{N}\sum_{i=1}^{N}x_{i}\right\|^{2}.
\end{equation*}
Instead, FedCEF relies only on the general contractive property (Definition \ref{def:q_contractive_compressor}), making it applicable to a broader class of compression schemes and unbounded gradient settings.
\end{remark}

\section{The Proposed Algorithm}

In this section, we detail the design of our proposed algorithm, FedCEF. The full procedure is summarized in Algorithm \ref{alg:fedcef}. The algorithm proceeds iteratively across communication rounds $t=0, 1, \dots, T-1$. To mitigate communication frequency, clients perform $K$ local updates in parallel within each round, a widely adopted practice in federated learning. However, this local computation inherently introduces the problem of client drift due to statistical heterogeneity across client data. Simultaneously correcting this drift under aggressive compression and non-smooth constraints is a central challenge addressed by our proposed algorithm.

\subsection{Decoupled Proximal Local Updates}
To accurately handle the non-smooth regularizer $h$, we employ a decoupled proximal update scheme. Each client $i$ maintains two distinct states: a pre-proximal model $\hat{x}_i$ and a post-proximal model $x_i$. In the $k$-th local step of round $t$, the client first updates the pre-proximal state using the stochastic gradient $g_i(x_i^{t,k})$ and the drift correction term\begin{equation} \label{eq:pre_prox_update}\hat{x}_i^{t,k+1}=\hat{x}_i^{t,k}-\alpha (g_i(x_i^{t,k})+c^t-c_i^t),
\end{equation}
where $c^t$ and $c_i^t$ denote the global and local control variates, respectively.

Subsequently, the proximal operator is applied to $\hat{x}_i^{t,k+1}$ to obtain the post-proximal model
\begin{equation} \label{eq:post_prox_update}
x_i^{t,k+1}=\textbf{prox}_{(k+1)\alpha h}(\hat{x}_i^{t,k+1}).\end{equation}

When preparing for communication after $K$ local steps, FedCEF utilizes the pre-proximal state $\hat{x}_i$ rather than the post-proximal $x_i$ as the basis for transmission. This decoupling is mathematically critical. Since the proximal operator is non-linear, naively transmitting and averaging the post-proximal models would induce approximation errors and distort the global gradient direction. In contrast, maintaining $\hat{x}_i$ acts as a linear accumulator, ensuring that the server can extract the true average gradient without distortion:$$\hat{x}_i^{t,k+1}=\hat{x}_i^{t,k}-\alpha (g_i(x_i^{t,k})+c^t-c_i^t),  $$
$$\frac{\hat{x}_i^{t,0}-\hat{x}_i^{t,K}}{\alpha K}+c_i^t-c^t=\frac{1}{K}\sum_{k=0}^{K-1}g_i(x_i^{t,k}).  $$
Furthermore, as shown in \eqref{eq:post_prox_update}, we employ a cumulative step size $(k+1)\alpha$ in the proximal step to further ensure structural consistency. This design choice ensures that the local model evolves in alignment with a centralized Proximal Gradient Descent (PGD) trajectory by matching the cumulative step size of the preceding $k+1$ gradient steps.

\subsection{Communication-Efficient Uplink and Downlink}

\textbf{Uplink Compression with Momentum.} 
Upon completing $K$ local steps, the client computes a momentum-based estimator $v_i^{t+1}$ to reduce the variance of the transmitted signal. Specifically, incorporating the historical information $v_i^t$, the update rule is given by
\begin{equation} \label{eq:momentum_update}
v_i^{t+1} = (1-\eta)v_i^t + \eta\left(\frac{\hat{x}_i^{t,0}-\hat{x}_i^{t,K}}{\alpha K} + c_i^t - c^t\right).
\end{equation}
This design aligns with recent findings that momentum effectively mitigates the negative impact of biased compression \cite{fatkhullin2023momentum}.

Next, to enable efficient uplink transmission, the client compresses the deviation vector $(v^{t+1}_i-c_i^t)$ using a general operator $\mathcal{C}(\cdot)$ (e.g., biased Top-$k$ sparsification), yielding
$\Delta_i^t = \mathcal{C}(v^{t+1}_i - c_i^t).$
The local control variable is immediately updated via the error feedback rule $c_i^{t+1} = c_i^t + \Delta_i^t$. The core theoretical advantage of this mechanism is that as the algorithm approaches a stationary point, the magnitude of the transmitted signal $\|\Delta_i^t\|$ gradually vanishes. Consequently, the compression error decays asymptotically:
\begin{equation*}
\mathbb{E}\|\mathcal{C}(\Delta_{i}^{t})-\Delta_{i}^{t}\|^{2} \leq q^2\|\Delta_{i}^{t}\|^{2} \to 0,
\end{equation*}
which eliminates the impact of quantization noise even in the presence of client heterogeneity.

\textbf{Downlink Reconstruction Strategy.} 
On the server side, aggregating the compressed updates yields the new global control $c^{t+1} = c^t + \frac{1}{N}\sum \Delta_i^t$. A standard control variate approach would require broadcasting both the model $z^{t+1}$ and the control $c^{t+1}$, doubling the downlink cost.

In contrast, FedCEF employs a pre-proximal broadcast strategy. The server transmits only the pre-proximal global iterate $\tilde{z}^{t+1} = z^t - \beta c^{t+1}$. Upon receiving $\tilde{z}^{t+1}$, clients perform two local operations:
\begin{enumerate}
    \item \textbf{Global Proximal Step:} Clients compute $z^{t+1} = \textbf{prox}_{\beta h}(\tilde{z}^{t+1})$. Since the regularizer $h$ is known locally, this ensures the non-smooth structure is imposed exactly without server-side computation.
    \item \textbf{Control Reconstruction:} Leveraging the linear relationship $\tilde{z}^{t+1} = z^t - \beta c^{t+1}$, clients exactly reconstruct the new global control variable
    \begin{equation*} \label{eq:reconstruction}
    c^{t+1} = \frac{z^t - \tilde{z}^{t+1}}{\beta}.
    \end{equation*}
\end{enumerate}
This mechanism ensures structural consistency while reducing the downlink communication overhead by half compared to standard methods.

\begin{algorithm}[!t]
	\makeatletter
	\renewcommand\footnoterule{%
		\kern-3\p@
		\hrule\@width.4\columnwidth
		\kern2.6\p@}
\makeatother
\caption{FedCEF }\label{alg:fedcef}
\begin{minipage}{\linewidth}
\begin{algorithmic}[1]
\STATE \textbf{Initialization:} step sizes $\alpha, \beta$, momentum parameter $\eta$, initial model $z^0$, controls $c_i^0=c^0=0$, local momentum $v_i^0=0$
\FOR {communication round $t=0,1,...,T-1$} 
\STATE \textbf{Client side:}
\FOR{client $i$ in parallel}
\STATE $\hat{x}_i^{t,0}=z^t$
\FOR{$k=0,1,...,K-1$}
\STATE Compute stochastic gradient $g_i(x_i^{t,k})$
\STATE Update $\hat{x}_i^{t,k+1}=\hat{x}_i^{t,k}-\alpha (g_i(x_i^{t,k})+c^t-c_i^t)$
\STATE Update $x_i^{t,k+1}=\textbf{prox}_{(k+1)\alpha h}(\hat{x}_i^{t,k+1})$
\ENDFOR
\STATE Update $v_i^{t+1}=(1-\eta)v_i^t+\eta(\frac{\hat{x}_i^{t,0}-\hat{x}_i^{t,K}}{\alpha K}+c_i^t-c^t)$
\STATE Send $\Delta_i^t=\mathcal{C}(v^{t+1}_i-c_i^t)$ to server
\STATE Update $c_i^{t+1}=c_i^t+\Delta_i^t$
\STATE \textbf{Downlink  Reconstruction:}
\STATE Receive pre-prox model $\tilde{z}^{t+1}$ from server
\STATE Reconstruct $c^{t+1} = (z^t - \tilde{z}^{t+1})/\beta$
\STATE Update $z^{t+1} = \textbf{prox}_{\beta h}(\tilde{z}^{t+1})$
\ENDFOR
\STATE \textbf{Server side:}
\STATE Receive $\Delta_i^t$ from all clients
\STATE Update $c^{t+1}=c^t+\frac{1}{N}\sum_{i=1}^N\Delta_i^t$
\STATE Update $\tilde{z}^{t+1} = z^t - \beta c^{t+1}$
\STATE Broadcast $\tilde{z}^{t+1}$ to all clients
\ENDFOR
\end{algorithmic}
\end{minipage}
\end{algorithm}
\subsection{Mechanism of Control Variates}
To fully characterize the role of the control variates, we analyze the local correction mechanism in \eqref{eq:pre_prox_update}. The control variables are designed to accumulate past gradient information. Specifically, in the absence of compression, the local control variable $c_i^{t+1}$ tracks the local update trajectory, approximating the local gradient direction via the momentum history:\begin{equation}c_{i}^{t+1} = c_{i}^{t} + \Delta_i^t \approx \frac{\eta}{K} \sum_{k=0}^{K-1} g_i(x_i^{t,k}) + (1-\eta)v_i^t.\end{equation}Crucially, since $c_i^t$ captures the local gradient bias while $c^t$ estimates the global gradient direction, incorporating the term $(c^t - c_i^t)$ effectively neutralizes the statistical heterogeneity. This mechanism shifts the local update direction towards the global optimum, ensuring convergence to a high-quality solution without requiring restrictive bounded heterogeneity assumptions.

Therefore, the control variates serve a dual purpose. Structurally, the correction term acts as a variance reduction mechanism to rectify client drift. Dynamically, the update rules ensure that the transmitted signal vanishes asymptotically, rendering the algorithm robust to biased compression. By synergizing this compression-tolerant mechanism—which simultaneously manages communication error and data heterogeneity—with the decoupled proximal framework for non-smooth objectives, FedCEF achieves provable convergence and high efficiency in the complex non-convex FCO setting.


\section{Theoretical Analysis}

In this section, we establish the convergence properties of the proposed FedCEF for non-convex composite optimization. Specifically, we address the analytical challenges posed by the intricate interplay of local updates, gradient tracking with biased compression, and the non-smooth proximal mapping.

To measure the progress of our algorithm, we utilize the \textit{proximal gradient mapping} as our optimality criterion. For a step size $\beta > 0$, the proximal gradient $G_{\beta}(z)$ is defined as
\begin{equation}
G_{\beta}(z) \triangleq \dfrac{1}{\beta} \left( z - \textbf{prox}_{\beta h} \{ z - \beta \nabla f(z) \} \right).
\end{equation}
It is well-known that $G_{\beta}(z) = \bm{0}$ if and only if $z$ is a stationary point of \eqref{eqfco}.  Moreover, the proximal gradient mapping is a standard metric for non-smooth composite optimization \cite{j2016proximal}.

\subsection{Key Lemmas}

To facilitate the analysis, we first define the auxiliary variable $U^t$ which measures the accumulated deviation of local updates from the global model $$U^t \triangleq \frac{1}{NK}\sum_{i=1}^N\sum_{k=0}^{K-1}\mathbb{E}\left\|x_{i}^{t,k+1}-z^{t}\right\|^{2}.$$With this definition in place, our analysis begins by establishing a one-step descent guarantee. The following lemma quantifies the objective value evolution between successive global rounds, capturing the errors introduced by the composite federated setting.

\begin{lemma}[One-step Descent]
Suppose that Assumptions 1--4 hold, then, we have
\begin{align}\label{eq10}
&\mathbb{E}\left[ F\left( z^{t+1} \right) \right]\nonumber\\
&\leq \mathbb{E}\left[ F\left( z^{t} \right)\right]   + \left( \frac{L}{2} - \frac{1}{2\beta} \right) \mathbb{E}\left\| z^{t+1} - z^{t} \right\|^2 \nonumber\\
&\quad+ 10 \beta L^{2} \mathbb{E} \left\| z^{t} - z^{t-1} \right\|^{2} +4\beta \mathbb{E}\left\|v^{t}-\nabla f\left(z^{t-1}\right)\right\|^{2} \nonumber\\
&\quad+ \frac{6\beta}{N} \sum_{i=1}^N \left( \mathbb{E} \left\| v_{i}^{t} - c_{i}^{t} \right\|^{2}  + \eta^{2} \mathbb{E} \left\| v_{i}^{t} - \nabla f_{i}(z^{t-1}) \right\|^{2}  \right) \nonumber\\
& \quad+\left( L\beta - \frac{1}{2} \right)\mathbb{E} \|G_{\beta}(z^{t})\|^{2}  + \frac{2(1+N)\eta^{2}\beta\sigma^{2}}{NKB} + 6\eta^{2}\beta L^{2}U^{t}.
\end{align}

\end{lemma}
\begin{IEEEproof}
See APPENDIX A.
\end{IEEEproof}

The descent inequality established in Lemma 1 relies heavily on the quality of the gradient estimators $v_i^t$ and $v^t$. To ensure convergence, the estimation bias introduced by the momentum updates must be effectively controlled. The following lemmas demonstrate that this estimation error does not accumulate uncontrollably but instead satisfies a contractive recurrence relation. Lemma 2 explicitly bounds the mean-squared error of the local estimator.

\begin{lemma}[Local Estimation Recursion]\label{lem2}
Suppose that Assumptions 1--4 hold, then, we have
\begin{align}\label{eq11}
&\frac{1}{N}\sum_{i=1}^N\mathbb{E}\|v_i^{t+1}-\nabla f_i(z^{t})\|^{2}\nonumber\\
&\leq (1-\frac{\eta}{2}) \frac{1}{N}\sum_{i=1}^N\mathbb{E}\left\|v_i^{t}-\nabla f\left(z^{t-1}\right)\right\|^{2}\nonumber\\
&\quad +\frac{2L^2}{\eta} \mathbb{E}  \left\| z^{t} - z^{t-1} \right\|^{2} +2\eta L^2 U^{t}+\frac{2\eta^{2}\beta\sigma^{2}}{KB}.
\end{align}
\end{lemma}
\begin{IEEEproof}
See APPENDIX B.
\end{IEEEproof}

Subsequently, Lemma 3 aggregates these local bounds to limit the global estimation bias.
\begin{lemma}[Global Estimation Recursion]\label{lem3}
Suppose that Assumptions 1--4 hold, then, we have
\begin{align}\label{eq12}
&\mathbb{E}\|v^{t+1}-\nabla f(z^{t})\|^{2} \nonumber\\
&\leq (1-\frac{\eta}{2}) \mathbb{E}\left\|v^{t}-\nabla f\left(z^{t-1}\right)\right\|^{2}+\frac{2L^2}{\eta} \mathbb{E}  \left\| z^{t} - z^{t-1} \right\|^{2} \nonumber\\
&\quad+2\eta L^2 U^{t}+\frac{2\eta^{2}\beta\sigma^{2}}{NKB} .
\end{align}
\end{lemma}
\begin{IEEEproof}
See APPENDIX C.
\end{IEEEproof}

Next, we bound the compression error term $\mathbb{E}\|v_i^t - c_i^t\|^2$ appearing in Lemma 1. Although biased compressors introduce quantization noise, the error-feedback mechanism ensures this noise remains controllable. Lemma 4 establishes a contractive recursion for this compression error, proving that it stabilizes over time and does not impede convergence.

\begin{lemma}[Compression Error Recursion]\label{lem4}
Suppose that Assumptions 1--4 hold, then, we have
\begin{align}\label{eq13}
&\frac{1}{N}\sum_{i=1}^N\mathbb{E}\|v_i^{t+1}-c_i^{t+1}\|^{2}\nonumber\\
&\leq \frac{q}{N} \sum_{i=1}^N \mathbb{E}  \left\| v_{i}^{t} - c_{i}^{t} \right\|^{2} + \left( 1 + \frac{2}{1 - q} \right) \eta^{2} q^{2} L^{2} U^{t}\nonumber\\
& \quad+ \frac{4\eta^{2} q^{2}}{1 - q} \frac{1}{N} \sum_{i=1}^N \mathbb{E}  \left\| \nabla f_{i}(z^{t-1}) - v_{i}^{t} \right\|^{2}  + \frac{2\eta^{2} q^{2} \sigma^{2}}{KB} \nonumber \\
& \quad+ \frac{4\eta^{2} L^{2} q^{2}}{1 - q} \mathbb{E} \left\| z^{t} - z^{t-1} \right\|^{2}  . 
\end{align}
\end{lemma}
\begin{IEEEproof}
See APPENDIX D.
\end{IEEEproof}

Finally, Lemma 5 bounds the local client drift $U^t$. It demonstrates that the deviation from the global model is explicitly controlled by the stationarity measure $\|G_\beta(z^t)\|^2$ and the accumulated estimation errors, scaled by the step sizes.
\begin{lemma}[Local Client Drift]\label{lem5}
Suppose that Assumptions 1--4 hold, then, we have
\begin{align}\label{eq14}
U^t &\leq 13.8\alpha^2K^2 \frac{1}{N}\sum_{i=1}^N \left( \mathbb{E}\|c_{i}^{t}-v_{i}^{t}\|^{2} + \mathbb{E}\|v_{i}^{t} -\nabla f_{i}(z^{t-1})\|^{2} \right. \nonumber \\
&\quad \left. + L^{2}\mathbb{E}\|z^t-z^{t-1}\|^{2} \right)+18.4K^{2}\alpha^{2} B_{h}^{2} + 9.2K\alpha^2\frac{\sigma^2}{b}\nonumber\\
&\quad+ 4.6K^{2}\alpha^{2}\mathbb{E}\|G_{\beta}(z^t)\|^{2} .
\end{align}
\end{lemma}
\begin{IEEEproof}
See APPENDIX E.
\end{IEEEproof}
\subsection{Main Convergence Result}

Having established the objective descent (Lemma 1) and bounded the associated error terms (Lemmas 2--5), we now derive the global convergence rate. To handle the coupling between the function value decrease and the recursive errors, we construct a Lyapunov potential function $\Psi^t$, which augments the objective with the error terms
\begin{align*}
 \Psi^t \triangleq &\mathbb{E}[F(z^t) - F^*] + c_1 \mathbb{E}\|v^t -  \nabla f(z^{t})\|^2 \\
\quad&+ c_2 \frac{1}{N}\sum_{i=1}^N \mathbb{E}\|v_i^t - c_i^t\|^2+ c_3 \frac{1}{N}\sum_{i=1}^N \mathbb{E}\|v_i^t - \nabla f_i(z^{t})\|^2,   
\end{align*}
where $\{c_i\}$ are weights chosen to neutralize the positive error terms in Lemma 1.

\begin{theorem}[Convergence Analysis]
Suppose Assumptions 1-4 hold. Let $\alpha$ and $\eta_g$ be the local and server learning rates, respectively, and define the effective global step size as $\beta \triangleq \alpha \eta_g K$. 
If the step sizes are chosen to satisfy the following conditions:
\begin{align}
\label{eq:step_size_condition} 
\beta \leq \frac{\min\{\eta^2,(1-q)^2\}}{25L} \quad \text{and} \quad \eta_g \ge \frac{\sqrt{16(1-q)^2+161\eta^2}}{5\eta(1-q)},
\end{align}
then the sequence generated by Algorithm 1 satisfies
\begin{align}
\frac{1}{T}\sum_{t=0}^{T-1}\mathbb{E}\|G_{\beta}(z^{t})\|^{2} \leq \frac{\Psi^{0}}{\gamma \beta T} + \mathcal{C}_{\text{stoc}} \frac{\sigma^2}{KB} + \mathcal{C}_{\text{approx}} \frac{L^2\beta^2}{\eta_g^2}B_h^2,
\end{align}
where $\gamma = 0.15$, and the coefficients are defined as
\begin{align*}
\mathcal{C}_{\text{stoc}} &= 6.7 \left(\frac{17\eta}{N}+\frac{14\eta^2}{1-q}+\frac{140\eta^3}{(1-q)^2}\right), \\
\mathcal{C}_{\text{approx}} &= \frac{18.4}{\gamma} \left(16+\frac{161\eta^2}{(1-q)^2}\right).
\end{align*}
\end{theorem}

\begin{IEEEproof}
See APPENDIX F.
\end{IEEEproof}
\begin{remark}[Analysis of Convergence and Residual Error]
Theorem 1 establishes that the algorithm converges to a neighborhood of a stationary point at a sublinear rate of $\mathcal{O}(1/T)$. We provide a detailed characterization of the \textit{residual error}, emphasizing its controllability:
\begin{itemize}
    \item \textbf{Statistical Variance Control:} The stochastic term is dominated by the variance $\sigma^2$. It can be suppressed by increasing the mini-batch size $B$ or the number of clients $N$, effectively reducing the noise floor of the stochastic gradient oracles.
    \item \textbf{Momentum-Stabilized Compression:} The error coefficients $\mathcal{C}_{\text{stoc}}$ and $\mathcal{C}_{\text{approx}}$ jointly capture the coupling between the momentum factor $\eta$ and the compression quality $q$. The analysis reveals a fundamental trade-off: aggressive compression amplifies both stochastic variance and approximation bias. This necessitates a corresponding reduction in $\eta$ (i.e., stronger momentum) to dampen the quantization noise, ensuring it does not dominate the convergence bound.
    \item \textbf{Decoupling Gap:} The final term involving $B_h^2$ arises from the decoupling of the proximal operator evaluation. This approximation gap is scaled by the global step size $\beta$, implying that the solution accuracy can be refined by employing a smaller or decaying global step size.
\end{itemize}
\end{remark}

\section{Experiments}
\label{sec:experiments}

In this section, we evaluate the performance of the proposed FedCEF algorithm on realistic image classification tasks. We compare FedCEF against state-of-the-art baselines to demonstrate its communication efficiency and convergence robustness, particularly under extreme compression ratios in non-convex federated composite optimization settings.

\subsection{Experimental Setup}

\textbf{Datasets and Models.}
We conduct experiments on two benchmark datasets: CIFAR-10 and MNIST, simulated in a non-IID federated setting.
For CIFAR-10, we utilize a custom 4-Layer CNN (approx. 4.4M parameters) consisting of four convolutional layers and two fully connected layers. The data is partitioned among $N=10$ clients using a Dirichlet distribution $\text{Dir}(0.6)$ to simulate significant data heterogeneity.
For MNIST, we use a lightweight CNN (approx. 0.4M parameters) consisting of two convolutional layers, with data partitioned using $\text{Dir}(0.5)$.
The global objective includes an $\ell_1$-regularization term ($\lambda=10^{-5}$) to induce structural sparsity.

\textbf{Baselines.}
We compare FedCEF against three methods:
\begin{itemize}
    \item \textbf{Algorithm \cite{zhang2026non}:} The standard uncompressed federated composite optimization algorithm, serving as the accuracy upper bound.
    \item \textbf{FedDA \cite{yuan2021federated}:} Federated Dual Averaging, which maintains dual states to handle regularization.
    \item \textbf{FedCanon \cite{zhou2025fedcanon}:} A canonical algorithm for non-smooth federated optimization.
\end{itemize}

\textbf{Implementation Details.}
We implement all algorithms in PyTorch.
For CIFAR-10, we set the total rounds $T=400$, local steps $K=30$, and batch size $B=64$. The learning rates are fixed at $\alpha=0.06$ (client) and $\eta_g=1.0$ (server).
For MNIST, we set $T=65$ and $K=10$. The learning rates are fixed at $\alpha=0.1$ (client) and $\eta_g=1.0$ (server).

\textbf{Compression Mechanism.}
FedCEF employs a Top-$k$ sparsification operator for uplink communication. To investigate the algorithm's robustness, we evaluate two compression ratios: moderate compression ($r=0.1$) and extreme compression ($r=0.01$), retaining only 10\% and 1\% of the gradient elements, respectively.
We strictly calculate the communication cost: sparse uplink updates consume 8 bytes per retained element (value + index), while the dense downlink broadcast consumes 4 bytes per element.

\subsection{Performance Evaluation on CIFAR-10}

We prioritize the analysis on CIFAR-10 due to its complexity. Fig. \ref{fig:cifar_acc} and Fig. \ref{fig:cifar_loss} illustrate the test accuracy and training loss against the accumulated communication cost.

\textbf{Communication Efficiency under Extreme Compression.}
FedCEF demonstrates superior communication efficiency, particularly in the low-bandwidth regime.
As shown in Fig. \ref{fig:cifar_acc}, while the uncompressed baseline (Algorithm \cite{zhang2026non}) consumes excessive bandwidth to converge, FedCEF with $r=0.01$ achieves a competitive accuracy of approximately 80\% while reducing the total communication volume by 49\% (72.79 GB vs. 142.72 GB) compared to the uncompressed baseline.
Compared to FedDA and FedCanon, FedCEF reaches the same loss level with significantly fewer transmitted bytes (Fig. \ref{fig:cifar_loss}).
This empirical success validates our theoretical design: the dual-purpose correction mechanism effectively mitigates the noise from aggressive quantization ($99\%$ sparsity) while simultaneously correcting the client drift caused by data heterogeneity. This prevents the divergence typically observed in naive sparsification methods applied to non-IID data.

\textbf{Accuracy and Robustness.}
It is worth noting that despite the aggressive compression, FedCEF ($r=0.01$) maintains a final accuracy comparable to the moderate compression setting ($r=0.1$) and the uncompressed baseline.
This robustness is particularly significant given the non-IID data partition ($\text{Dir}(0.6)$). It confirms that our specific update dynamics ensure the transmitted error signals vanish asymptotically, thereby eliminating compression errors over time as theoretically predicted.

\begin{figure}[!t]
  \centering
  \includegraphics[width=0.4\textwidth]{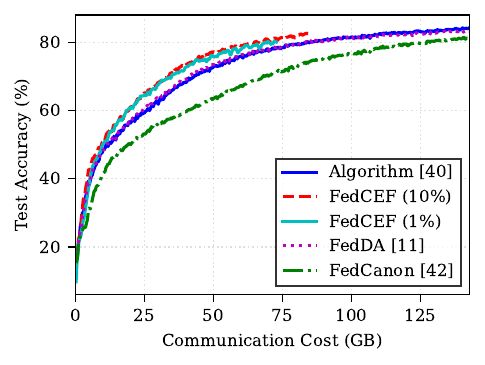}
  \caption{Test Accuracy vs. Communication Cost on CIFAR-10. FedCEF ($r=0.01$) achieves target accuracy with minimal bandwidth consumption.}
  \label{fig:cifar_acc}
\end{figure}

\begin{figure}[!t]
  \centering
  \includegraphics[width=0.4\textwidth]{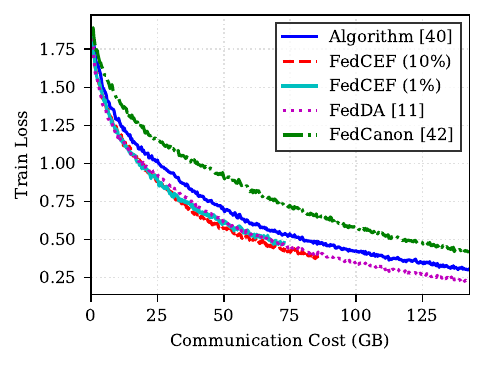}
  \caption{Train Loss vs. Communication Cost on CIFAR-10. FedCEF demonstrates superior communication efficiency in loss reduction.}
  \label{fig:cifar_loss}
\end{figure}

\subsection{Performance on MNIST}

We further validate FedCEF on MNIST to assess generalizability across model architectures.
As observed in Fig. \ref{fig:mnist_acc} and Fig. \ref{fig:mnist_loss}, FedCEF consistently outperforms the baselines.
Even with the lightweight CNN model where parameter redundancy is lower, the advantage of sparse communication remains significant. All algorithms eventually converge to high accuracy ($>98\%$), but FedCEF reaches the convergence plateau with the lowest communication overhead. This reinforces that the decoupled proximal update is robust to different loss landscapes.
\begin{figure}[!t]
  \centering
  \includegraphics[width=0.4\textwidth]{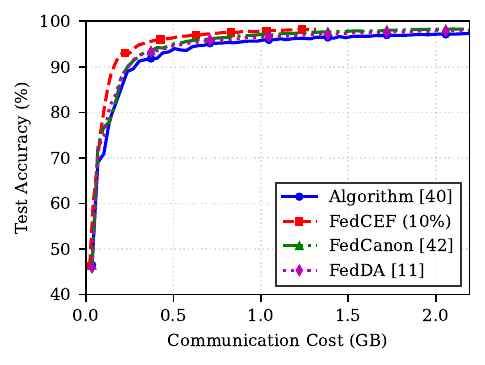}
  \caption{Test Accuracy vs. Communication Cost on MNIST.}
  \label{fig:mnist_acc}
\end{figure}

\begin{figure}[!t]
  \centering
  \includegraphics[width=0.4\textwidth]{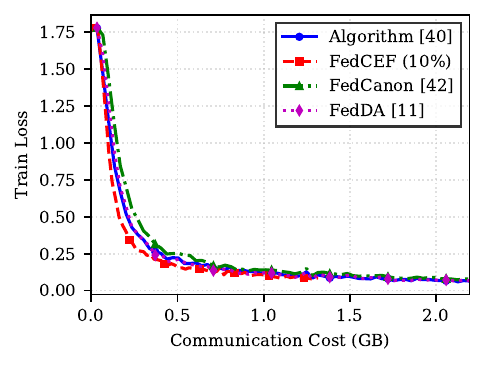}
  \caption{Train Loss vs. Communication Cost on MNIST.}
  \label{fig:mnist_loss}
\end{figure}

\section{Conclusion}
We propose FedCEF, a novel and communication-efficient algorithm tailored for non-convex Federated Composite Optimization (FCO) under significant statistical heterogeneity. FedCEF effectively addresses the challenges of non-smooth regularization by employing a decoupled proximal update scheme, which separates proximal operator evaluations from the communication process to preserve the learned model's structural integrity. To simultaneously mitigate communication bottlenecks and client drift, the algorithm integrates a dual-purpose correction mechanism based on control variates and momentum-based variance reduction. Our specific update dynamics ensure that the transmitted signals vanish asymptotically, rendering the algorithm uniquely robust to biased compression and effectively eliminating compression errors even in the presence of extreme data heterogeneity. FedCEF achieves sublinear convergence rates and high communication efficiency in composite settings without requiring restrictive assumptions on data similarity. Experimental validation demonstrates FedCEF's competitive performance against uncompressed baselines, with communication efficiency improvements.


\appendices

\section{Proof of Lemma 1}
For brevity, we use the abbreviated summations $\sum_i$ and $\sum_k$ to denote $\sum_{i=1}^N$ and $\sum_{k=0}^{K-1}$, respectively. The averaged local stochastic gradients and their global average are defined as
\begin{align*}
g_{i}^{t} \triangleq \frac{1}{K}\sum_{k=0}^{K-1}g_{i}^{t,k}, \quad 
g^{t} \triangleq \frac{1}{N}\sum_{i=1}^{N}g_{i}^{t}.
\end{align*}
Similarly, we denote the average momentum term by $v^t \triangleq \frac{1}{N}\sum_{i=1}^{N}v_i^t$. Additionally, we note that $c^{t} \equiv \frac{1}{N}\sum_{i=1}^{N}c_{i}^{t}$ holds for all $t \geq 0$ in our algorithm.

Recall that in FedCEF, with $\beta \triangleq \eta_{g}\alpha K$, the global model is updated via the proximal mapping
\[
z^{t+1} = \textbf{prox}_\beta \big( z^{t} - \beta\boldsymbol{v}^t \big),
\]
where $\boldsymbol{v}^t = c^t+\frac{1}{N}\sum_{i=1}^{N}\mathcal{C}_{i}\big(u_{i}^{t+1}-c_{i}^{t}\big)$. To analyze the convergence of this update rule, we start from a fundamental inequality that characterizes the descent property of the proximal mapping for $L$-smooth composite functions. Specifically, let $\mathbf{x}^+ := P_\eta(\mathbf{x} - \eta \boldsymbol{v})$, then for any $\mathbf{z} \in \mathbb{R}^d$, we have
\begin{align}
&F\left(\mathbf{x}^+\right) \leq F(\mathbf{z}) + \left\langle \nabla F(\mathbf{x}) - \boldsymbol{v}, \mathbf{x}^+ - \mathbf{z} \right\rangle \nonumber\\
&- \frac{1}{\eta} \left\langle \mathbf{x}^+ - \mathbf{x}, \mathbf{x}^+ - \mathbf{z} \right\rangle + \frac{L}{2} \left\| \mathbf{x}^+ - \mathbf{x} \right\|^2 + \frac{L}{2} \left\| \mathbf{z} - \mathbf{x} \right\|^2, \ \forall \mathbf{z}.    \nonumber
\end{align}

We apply this fundamental inequality to two specific transitions: the virtual centralized step $\tilde{\mathbf{x}}^{t+1}$ and the actual aggregated update $z^{t+1}$. 

First, setting $\mathbf{x}^+ = \tilde{\mathbf{x}}^{t+1}$, $\mathbf{x} = \mathbf{z} = z^{t}$ and $\boldsymbol{v} = \nabla f(z^{t})$ leads to the descent of the virtual sequence
\begin{align}
&\mathbb{E}\left[ F\left( \tilde{\mathbf{x}}^{t+1} \right) \right] 
\leq \mathbb{E}\left[ F\left( z^{t} \right) \right] \nonumber\\
& + \left( \frac{L}{2} - \frac{1}{2\beta} \right) \left\| \tilde{\mathbf{x}}^{t+1} - z^{t} \right\|^2 - \frac{1}{2\beta} \left\| \tilde{\mathbf{x}}^{t+1} - z^{t} \right\|^2. 
\end{align}

Next, setting $\mathbf{x}^+ = z^{t+1}$, $\mathbf{x} = z^{t}$, $\mathbf{z} = \tilde{\mathbf{x}}^{t+1}$, and $\boldsymbol{v} = \boldsymbol{v}^t$ gives the descent of the actual update
\begin{align}
&\mathbb{E}\left[ F\left( z^{t+1} \right) \right] \nonumber\\
&\leq \mathbb{E}\left[ F\left( \tilde{\mathbf{x}}^{t+1} \right) + \left\langle \nabla f\left( z^{t} \right) - \boldsymbol{v}^t, z^{t+1} - \tilde{\mathbf{x}}^{t+1} \right\rangle \right. \nonumber\\
&\quad - \frac{1}{\beta} \left\langle z^{t+1} - z^{t}, z^{t+1} - \tilde{\mathbf{x}}^{t+1} \right\rangle \nonumber \\
&\quad + \frac{L}{2} \left\| z^{t+1} - z^{t} \right\|^2 + \frac{L}{2} \left\| \tilde{\mathbf{x}}^{t+1} - z^{t} \right\|^2 \Bigg] \nonumber \\
&= \mathbb{E}\left[ F\left( \tilde{\mathbf{x}}^{t+1} \right) + \left\langle \nabla f\left( z^{t} \right) - \boldsymbol{v}^t, z^{t+1} - \tilde{\mathbf{x}}^{t+1} \right\rangle \right. \nonumber \\
&\quad + \left( \frac{L}{2} - \frac{1}{2\beta} \right) \left\| z^{t+1} - z^{t} \right\|^2 \nonumber \\
&\quad+ \left( \frac{L}{2} + \frac{1}{2\beta} \right) \left\| \tilde{\mathbf{x}}^{t+1} - z^{t} \right\|^2  \left. - \frac{1}{2\beta} \left\| z^{t+1} - \tilde{\mathbf{x}}^{t+1} \right\|^2 \right]. 
\end{align}

where the last equation uses $2\langle b, b - a \rangle = \|b - a\|^2 + \|b\|^2 - \|a\|^2$. Evaluating the actual update $z^{t+1}$ relative to the virtual point $\mathbf{z} = \tilde{\mathbf{x}}^{t+1}$, we have
\begin{align}\label{eq1}
&\mathbb{E}\left[ F\left( z^{t+1} \right) \right]\nonumber\\
&\leq \mathbb{E}\Bigg[ F\left( z^{t} \right)   + \left( \frac{L}{2} - \frac{1}{2\beta} \right) \left\| z^{t+1} - z^{t} \right\|^2 \nonumber\\
&\quad+ \left( L - \frac{1}{2\beta} \right) \left\| \tilde{\mathbf{x}}^{t+1} - z^{t} \right\|^2- \underbrace{\frac{1}{2\beta} \left\| z^{t+1} - \tilde{\mathbf{x}}^{t+1} \right\|^2}_{(\text{I})} \nonumber \\
&\quad+ \underbrace{\left\langle z^{t+1} - \tilde{\mathbf{x}}^{t+1}, \nabla f\left( z^{t} \right) - \boldsymbol{v}^t \right\rangle}_{(\text{II})} \Bigg]. 
\end{align}

The negative quadratic term (I) allows us to absorb the cross-term (II) via a specialized form of Young's inequality, specifically $\|a + b\|^2 \leq \frac{1}{2\beta}\|a\|^2 + \frac{\beta}{2}\|b\|^2$,
\begin{align}
&(\text{I}) + (\text{II})\nonumber\\
&\leq (\text{I}) + \frac{1}{2\beta} \left\| z^{t+1} - \tilde{\mathbf{x}}^{t+1} \right\|^2 + \frac{\beta}{2} \left\| \nabla f\left( z^{t} \right) - \boldsymbol{v}^t \right\|^2 \nonumber\\
&= \frac{\beta}{2} \left\| \nabla f\left( z^{t} \right) - \boldsymbol{v}^t \right\|^2. 
\end{align}

To derive the upper bound, we introduce the virtual update without compression $\tilde{\boldsymbol{v}}^t$. Therefore, the (5) can be bounded as
 \begin{align}\label{eq2}
& \frac{\beta}{2}\mathbb{E} \left\| \nabla f\left( z^{t} \right) - \boldsymbol{v}^t \right\|^2\nonumber\\
& \leq \underbrace{\beta \mathbb{E} \left\|  \nabla f\left( z^{t} \right) -  \tilde{\boldsymbol{v}}^t \right\|^2}_{(\text{III})}+\underbrace{\beta \mathbb{E}\left\| \boldsymbol{v}^t -  \tilde{\boldsymbol{v}}^t \right\|^2}_{(\text{IV})}.
\end{align}

First, we try to bound the term (\text{III}) with the smoothness assumption
\begin{align}\label{eq3}
&\beta \mathbb{E}\|\nabla f\left( z^{t} \right) -  \tilde{\boldsymbol{v}}^t \|^2 \nonumber\\
&=\beta \mathbb{E} \left\| c^t + \frac{1}{N} \sum_{i=1}^{N}  (u_i^{t+1} - c_i^t) - \nabla f(z^{t}) \right\|^2 \nonumber \\
&=\beta \mathbb{E} \left\| (1-\eta)(v^t-\nabla f(z^{t})) +\eta(g^t- \nabla f(z^{t}) )\right\|^2 \nonumber\\
&\leq 2\beta \mathbb{E}\left\|v^{t}-\nabla f\left(z^t\right)\right\|^{2}+3\eta^{2}\beta  U^{t}+\frac{2\eta^{2}\beta\sigma^{2}}{NKB} \nonumber\\
&\leq 4\beta \mathbb{E}\left\|v^{t}-\nabla f\left(z^{t-1}\right)\right\|^{2}+4\beta L^{2}\mathbb{E}\left\|z^{t}-z^{t}\right\|^{2}\nonumber\\
&\quad+3\eta^{2}\beta L^{2}U^{t}+\frac{2\eta^{2}\beta \sigma^{2}}{NKB}.
\end{align}

Next, we try to quantify the distance between compressed update and virtual update with the Asssumption 3
\begin{align}\label{eq4}
&\mathbb{E}\|\boldsymbol{v}^t -  \tilde{\boldsymbol{v}}^t \|^2 \nonumber\\
&= \mathbb{E}\left\|\frac{1}{N}\sum_{i}\mathcal{C}_{i}\left(u_{i}^{t+1}-c_{i}^{t}\right)-\left(u_{i}^{t+1}-c_{i}^{t}\right)\right\|^{2} \nonumber\\
&\leq \frac{q^{2}}{N}\sum_{i}\mathbb{E}\left\|u_{i}^{t+1}-c_{i}^{t}\right\|^{2}\nonumber\\
&= \frac{q^{2}}{N}\sum_{i}\mathbb{E}\left\|v_{i}^{t}+\eta\left(g_{i}^{t}-v_{i}^{t}\right)-c_{i}^{t}\right\|^{2}\nonumber\\
&= \frac{q^2}{N} \sum_{i} \mathbb{E}  \left\| v_{i}^{t} + \eta \left( \nabla f_{i}(z^{t}) - \nabla f_{i}(z^{t-1}) \right) \right.  \nonumber \\
&\quad  \left. + \eta (\nabla f_{i}(z^{t-1}) - v_{i}^{t}) + \eta (g_{i}^{t} - \nabla f_{i}(z^{t})) - c_{i}^{t} \right\|^{2} \nonumber\\
&\leq \frac{2q^2}{N} \sum_{i} \mathbb{E} \left\| v_{i}^{t} + \eta (\nabla f_{i}(z^{t}) - \nabla f_{i}(z^{t-1}))\right.  \nonumber\\
&\quad  \left.+ \eta (\nabla f_{i}(z^{t-1}) - v_{i}^{t}) - c_{i}^{t} \right\|^{2}  + 3\eta^{2}L^{2}U^{t} + \frac{2\eta^{2}\sigma^{2}}{KB} \nonumber\\
&\leq \frac{6q^2}{N} \sum_{i} \left( \mathbb{E}  \left\| v_{i}^{t} - c_{i}^{t} \right\|^{2}  + \eta^{2} \mathbb{E}  \left\| v_{i}^{t} - \nabla f_{i}(z^{t-1}) \right\|^{2}  \right. \nonumber\\
&\left. \quad+ \eta^{2}L^{2} \mathbb{E} \left\| z^{t} - z^{t-1} \right\|^{2}  \right) + 3\eta^{2}L^{2}U^{t} + \frac{2\eta^{2}\sigma^{2}}{KB}.
\end{align}

Substituting equations (\ref{eq3}) and (\ref{eq4}) into (\ref{eq2}), and then substituting the result into (\ref{eq1}), yields Lemma 1.

\section{Proof of Lemma 2}


\begin{IEEEproof}
We begin by substituting the update rule of $v_i^{t+1}$ into the error term. By applying Young’s inequality and the variance bound (Assumption 3), along with the $L$-smoothness of the loss function(Assumption 1), we derive the recursion
\begin{align}
&\frac{1}{N}\sum_{i=1}^N\mathbb{E}\|v_i^{t+1}-\nabla f_i(z^{t})\|^{2}\nonumber\\
&\leq \frac{1}{N}\sum_{i=1}^N\mathbb{E} \left\| (1-\eta)(v_i^t-\nabla f_i(z^{t}))+\eta(g_i^t- \nabla f_i(z^{t}) )\right\|^2 \nonumber\\
&\leq (1-\eta) \frac{1}{N}\sum_{i=1}^N\mathbb{E}\left\|v_i^{t}-\nabla f_i\left(z^t\right)\right\|^{2}+2\eta L^2 U^{t}+\frac{2\eta^{2}\beta\sigma^{2}}{KB} \nonumber\\
&\leq (1-\frac{\eta}{2}) \frac{1}{N}\sum_{i=1}^N\mathbb{E}\left\|v_i^{t}-\nabla f_i\left(z^{t-1}\right)\right\|^{2}\nonumber\\
&\quad+\frac{2L^2}{\eta} \mathbb{E}  \left\| z^{t} - z^{t-1} \right\|^{2} +2\eta L^2 U^{t}+\frac{2\eta^{2}\beta\sigma^{2}}{KB}.
\end{align}
Rearranging the terms yields the desired inequality, which completes the proof.
\end{IEEEproof}
\section{Proof of Lemma 3}

\begin{IEEEproof}
Following a similar logic to Lemma \ref{lem2}, we analyze the global estimator $v^{t+1}$. By substituting the global update rule and exploiting the independence of stochastic noise across $N$ clients (which reduces the variance by a factor of $1/N$), we obtain
\begin{align}
&\mathbb{E}\|v^{t+1}-\nabla f(z^{t})\|^{2}\nonumber\\
&\leq \mathbb{E} \left\| (1-\eta)(v^t-\nabla f(z^{t})) +\eta(g^t- \nabla f(z^{t} ))\right\|^2 \nonumber\\
&\leq (1-\eta) \mathbb{E}\left\|v^{t}-\nabla f\left(z^t\right)\right\|^{2}+2\eta L^2 U^{t}+\frac{2\eta^{2}\beta\sigma^{2}}{NKB} \nonumber\\
&\leq (1-\frac{\eta}{2}) \mathbb{E}\left\|v^{t}-\nabla f\left(z^{t-1}\right)\right\|^{2} +2\eta L^2 U^{t}\nonumber\\
&\quad+\frac{2L^2}{\eta} \mathbb{E}  \left\| z^{t} - z^{t-1} \right\|^{2}+\frac{2\eta^{2}\beta\sigma^{2}}{NKB}.
\end{align}
This establishes the recursion bound for the global estimator and completes the proof.
\end{IEEEproof}


\section{Proof of Lemma 4}
\begin{IEEEproof}
Using the Definition 1 and $u_{i}^{t+1}=v_i^t+\eta(g_i^t-v_i^t)$, we have
\begin{align}\label{eq5}
&\mathbb{E}\left\|v_{i}^{t+1}-c_{i}^{t+1}\right\|^{2}\nonumber\\
&= \frac{1}{N} \mathbb{E}\left\|u_{i}^{t+1}-\mathcal{C}_{i}\left(u_{i}^{t+1}-c_{i}^{t}\right)-c_{i}^{t}\right\|^{2} \nonumber\\
&\leq  \frac{q^{2}}{N} \sum_{i} \mathbb{E}\left\|u_{i}^{t+1}-c_{i}^{t}\right\|^{2} \nonumber\\
&=  \frac{q^{2}}{N} \sum_{i} \mathbb{E}\left\|v_{i}^{t}+\beta\left(g_{i}^{t}-v_{i}^{t}\right)-c_{i}^{t}\right\|^{2}.
\end{align}

Next, we decompose the stochastic gradient $g_i^t$ into its conditional expectation and zero-mean noise. By leveraging the independence of the stochastic noise and Assumption 3 (bounded variance), we separate the variance term
\begin{align}
&\frac{q^{2}}{N} \sum_{i} \mathbb{E}  \left\| v_{i}^{t} + \eta (g_{i}^{t} - v_{i}^{t}) - c_{i}^{t} \right\|^{2} \nonumber\\
&= \frac{q^{2}}{N} \sum_{i} \mathbb{E}  \left\| v_{i}^{t} + \eta (\nabla f_{i}(z^{t}) - v_{i}^{t}) + \eta (g_{i}^{t} - \nabla f_{i}(z^{t})) - c_{i}^{t} \right\|^{2}  \nonumber \\
&= \frac{q^{2}}{N} \sum_{i} \Bigg( \mathbb{E}  \left\| v_{i}^{t} - c_{i}^{t} + \eta (\nabla f_{i}(z^{t}) - v_{i}^{t}) \right\|^{2}  \nonumber \\
&\quad + 2\eta \mathbb{E} \Bigg[ \bigg\langle v_{i}^{t} - c_{i}^{t} + \eta (\nabla f_{i}(z^{t}) - v_{i}^{t}), \nonumber \\
&\quad\quad \frac{1}{K} \sum_{k} \nabla f_{i}(x_{i}^{t,k}) - \nabla f_{i}(z^{t}) \bigg\rangle \Bigg] \nonumber \\
&\quad + \eta^{2} \mathbb{E}  \left\| \frac{1}{K} \sum_{k} \nabla F(x_{i}^{t,k}; \xi_{i}^{t,k}) - \nabla f_{i}(z^{t}) \right\|^{2} \Bigg) \nonumber \\
&\leq \frac{q^{2}}{N} \sum_{i} \Bigg( \mathbb{E}  \left\| v_{i}^{t} - c_{i}^{t} + \eta (\nabla f_{i}(z^{t}) - v_{i}^{t}) \right\|^{2}  \nonumber \\
&\quad + 2\eta \mathbb{E} \Bigg[ \bigg\langle v_{i}^{t} - c_{i}^{t} + \eta (\nabla f_{i}(z^{t}) - v_{i}^{t}), \nonumber \\
&\quad\quad \frac{1}{K} \sum_{k} \nabla f_{i}(x_{i}^{t,k}) - \nabla f_{i}(z^{t}) \bigg\rangle \Bigg] \nonumber \\
&\quad + 2\eta^{2} \mathbb{E}  \left\| \frac{1}{K} \sum_{k} \nabla f_{i}(x_{i}^{t,k}) - \nabla f_{i}(z^{t}) \right\|^{2}  + \frac{2\eta^{2} \sigma^{2}}{KB} \Bigg) \nonumber \\
&\leq \frac{q^{2}}{N} \sum_{i} \mathbb{E}  \bigg\| v_{i}^{t} - c_{i}^{t} + \eta (\nabla f_{i}(z^{t}) - v_{i}^{t})\nonumber\\
&\quad + \frac{\eta}{K} \sum_{k} (\nabla f_{i}(x_{i}^{t,k}) - \nabla f_{i}(z^{t})) \bigg\|^{2} + \eta^{2} q^{2} L^{2} U^{t} \nonumber \\
&\quad  + \frac{2\eta^{2} q^{2} \sigma^{2}}{KB}.
\end{align}
By further using Sedrakyan's inequality and Assumption 1, we obtain
\begin{align}\label{eq6}
&\frac{q^{2}}{N} \sum_{i} \mathbb{E}  \left\| v_{i}^{t} + \eta (g_{i}^{t} - v_{i}^{t}) - c_{i}^{t} \right\|^{2} \nonumber\\
&\leq \frac{q^{2}}{N} \sum_{i} \left( \frac{1}{q} \mathbb{E} \left\| v_{i}^{t} - c_{i}^{t} \right\|^{2}  + \frac{2\eta^{2}}{1 - q} \mathbb{E}  \left\| \nabla f_{i}(z^{t}) - v_{i}^{t} \right\|^{2}  \right. \nonumber \\
&\quad \left. + \frac{2\eta^{2}}{1 - q} \mathbb{E}  \left\| \frac{1}{K} \sum_{k} \nabla f_{i}(x_{i}^{t,k}) - \nabla f_{i}(z^{t}) \right\|^{2}  \right) \nonumber \\
&\quad + \eta^{2} q^{2} L^{2} U^{t} + \frac{2\eta^{2} q^{2} \sigma^{2}}{KB} \nonumber \\
&\leq \frac{q^{2}}{N} \sum_{i} \left( \frac{1}{q} \mathbb{E}  \left\| v_{i}^{t} - c_{i}^{t} \right\|^{2}  + \frac{4\eta^{2}}{1 - q} \mathbb{E} \left\| v_{i}^{t} - \nabla f_{i}(z^{t-1}) \right\|^{2}  \right. \nonumber \\
&\quad + \frac{4\eta^{2} L^{2}}{1 - q} \mathbb{E}  \left\| z^{t} - z^{t-1} \right\|^{2}  \left. + \frac{2\eta^{2} L^{2}}{1 - q} \frac{1}{K} \sum_{i} \mathbb{E}  \left\| x_{i}^{t,k} - z^{t} \right\|^{2}  \right) \nonumber \\
&\quad + \eta^{2} q^{2} L^{2} U^{t} + \frac{2\eta^{2} q^{2} \sigma^{2}}{KB} \nonumber \\
&\leq \frac{q}{N} \sum_{i} \mathbb{E}  \left\| v_{i}^{t} - c_{i}^{t} \right\|^{2}   + \frac{4\eta^{2} L^{2} q^{2}}{1 - q} \mathbb{E}  \left\| z^{t} - z^{t-1} \right\|^{2} \nonumber\\
&\quad+ \frac{4\eta^{2} q^{2}}{1 - q} \frac{1}{N} \sum_{i} \mathbb{E}  \left\| \nabla f_{i}(z^{t-1}) - v_{i}^{t} \right\|^{2}  \nonumber \\
&\quad + \left( 1 + \frac{2}{1 - q} \right) \eta^{2} q^{2} L^{2} U^{t} + \frac{2\eta^{2} q^{2} \sigma^{2}}{KB}. 
\end{align}

Combining (\ref{eq5}) and (\ref{eq6}),we can finish the proof.
\end{IEEEproof}

\section{Proof of Lemma 5}
Given Lemma 1-4, we need to bound the local update term $U^t$.

\begin{IEEEproof}
If $K=1$, the local update is trivial ($x_{i,k}^t - z^t = 0_d$). For $K\geq2$, the accumulated local update in Algorithm 1 yields
\[
\widehat{x}_{i,k+1}^{t}=z^t-\alpha\sum_{\ell=0}^{k}\left(g_{i,l}^t+c^t-c_{i}^{t}\right).
\]
Let $x_{i,k+1}^{t}=\text{prox}_{(k+1)\alpha h}(\widehat{x}_{i,k+1}^{t})$. Based on the recursive update in Algorithm 1, we have
\begin{align}\label{eq7}
&\mathbb{E}\left\|x_{i,k+1}^{t}-z^t\right\|^{2} \nonumber \\
&=\mathbb{E}\left\|\text{prox}_{(k+1)\alpha h}\left(z^t-\alpha\sum_{\ell=0}^{k}\left(g_i(x_i^{t,\ell})+c^t-c_{i}^{t}\right)\right)-z^t\right\|^{2}. 
\end{align}

To quantify the deviation caused by the accumulated stochastic gradients and correction terms, we compare the actual update in (\ref{eq7}) against a theoretical reference iterate $\tilde{x}_{\mathrm{pgd}}$, which represents an exact proximal gradient step using the full gradient
\[
\tilde{x}_{\mathrm{pgd}}:=\text{prox}_{(k+1)\alpha h}\left(z^t-(k+1)\alpha\nabla f(z^t)\right).
\]


To simplify the notation, let us define the accumulated pre-proximal update as $y_{i,k}^t \triangleq z^t-\alpha\sum_{\ell=0}^{k}\left(g_i(x_i^{t,\ell})+c^t-c_{i}^{t}\right)$. By adding and subtracting $\tilde{x}_{\mathrm{pgd}}$, we decompose the error\begin{align}\label{eq8}
&\mathbb{E}\left\|x_{i,k+1}^{t}-z^t\right\|^{2} \nonumber\\ 
&= \E\big\| \text{prox}_{(k+1)\alpha h}\big(y_{i,k}^t\big) -\tilde{x}_{\text{pgd}} + \tilde{x}_{\text{pgd}}-z^t \big\|^{2} \nonumber\\
&\leq 2\underbrace{\E\left\| \text{prox}_{(k+1)\alpha h}\big(y_{i,k}^t\big)-\tilde{x}_{\text{pgd}} \right\|^{2}}_{(\text{V})}+ 2\underbrace{\E\left\| \tilde{x}_{\text{pgd}}-z^t \right\|^{2}}_{(\text{VI})}\nonumber \\
\end{align}
\noindent\textbf{Bounding (VI):} Let us characterize the updates using the generalized gradient mapping notation $G_{\eta}(x) \triangleq \frac{1}{\eta}(x - \text{prox}_{\eta h}(x - \eta \nabla f(x)))$. Based on the optimality conditions of the proximal operator, the updates for the virtual iterate $\tilde{x}_{\text{pgd}}$ and the global virtual iterate $\tilde{x}^{k+1}$ can be rewritten as
\begin{align*}
    \tilde{x}_{\text{pgd}} - z^t &= -(k+1)\alpha G_{(k+1)\alpha}(z^t), \\
    \tilde{x}^{k+1} - z^t &= -\beta G_{\beta}(z^t).
\end{align*}
Substituting these expressions into (VI) reveals the discrepancy caused by the step-size difference. By adding and subtracting the term $(k+1)\alpha G_{\beta}(z^t)$ inside the norm and applying the inequality $\|a+b\|^2 \leq 2\|a\|^2 + 2\|b\|^2$, we obtain
\begin{align}
    \text{(VI)} &= 2(k+1)^2\alpha^2 \mathbb{E}\big\|G_{(k+1)\alpha}(z^t) - G_{\beta}(z^t) + G_{\beta}(z^t)\big\|^2 \nonumber \\
    &\leq 4(k+1)^2\alpha^2 \mathbb{E}\big\|G_{(k+1)\alpha}(z^t) - G_{\beta}(z^t)\big\|^2 \nonumber \\
    &\quad + 4(k+1)^2\alpha^2 \mathbb{E}\big\|G_{\beta}(z^t)\big\|^2.
\end{align}
Recall that $G_{\eta}(x) \in \nabla f(x) + \partial g(x^+)$ for some $x^+$. Thus, the difference in the first term is bounded entirely by the variation in subgradients. Using Assumption 2 (bounded subgradients, i.e., $\|\partial g(x)\| \leq B_h$) and $k+1 \leq K$, we have $\|G_{(k+1)\alpha}(z^t) - G_{\beta}(z^t)\|^2 \leq (2B_h)^2 = 4B_h^2$. This yields
\begin{align}
    \text{(VI)} \leq 16 K^2\alpha^2 B_h^2 + 4K^2\alpha^2 \mathbb{E}\|G_{\beta}(z^t)\|^2.
\end{align}
\noindent\textbf{Bounding (V):} Using the non-expansiveness of the proximal operator and the definition of $\tilde{x}_{\mathrm{pgd}}$, we have
\begin{align}\label{eq9}
\text{(V)} &\leq  2\mathbb{E}\left\|\alpha\sum_{\ell=0}^{k}\left(g_{i,l}^t+c^t-c_{i}^{t}-\nabla f\left(z^t\right)\right)\right\|^{2}\nonumber\\
&\leq 4\underbrace{\mathbb{E}\left\|\alpha\sum_{\ell=0}^{k}\left(g_{i,l}^t-\nabla f_i\left(z^t\right)\right)\right\|^{2}}_{(\text{VII})}\nonumber \\
&\quad+ 4\underbrace{\mathbb{E}\left\|\alpha\sum_{\ell=0}^{k}\left(c^t
-\nabla f\left(z^t\right)-c_i^t+\nabla f_i\left(z^t\right)\right)\right\|^{2}}_{(\text{VIII})}.
\end{align}

For term (VII), we decompose the error into stochastic variance and path deviation. Using Assumptions 1 and 3
\begin{align}
\text{(VII)} &\leq 8\mathbb{E}\left\|\alpha\sum_{\ell=0}^{k}\left(g_{i,l}^t-\nabla f_i\left(x_{i,l}^t\right)\right)\right\|^{2}\nonumber\\
&\quad +8 \mathbb{E}\left\|\alpha\sum_{\ell=0}^{k}\left(\nabla f_i\left(x_{i,l}^t\right)-\nabla f_i\left(z^t\right)\right)\right\|^{2}\nonumber \\
&\leq 8K\alpha^2\frac{\sigma^2}{B} + 8(k+1)\alpha^2L^2\sum_{\ell=0}^{k}\mathbb{E}\left\|\left(x_{i,l}^t-z^t\right)\right\|^{2}.
\end{align}

For the correction error term, we exploit the relationship between the global control variate $c^t$ and the local ones $c_i^t$. Using the convexity of the squared norm (Jensen's inequality), we have $\mathbb{E}\|c^t - \nabla f(z^t)\|^2 \leq \frac{1}{N}\sum \mathbb{E}\|c_i^t - \nabla f_i(z^t)\|^2$. This allows us to bound the global error by the average of local errors. By further decomposing the local error terms, we obtain
\begin{align}
\text{(VIII)} &= 4\alpha^2 (k+1)^2 \mathbb{E}\| (c^t - \nabla f(z^t)) - (c_i^t - \nabla f_i(z^t)) \|^2 \nonumber \\
&\leq 12\alpha^2K^2\mathbb{E}\|c_{i}^{t}-\nabla f_i\left(z^t\right)\|^{2} \nonumber\\
& \leq 12\alpha^2K^2\left(\mathbb{E}\|c_{i}^{t}-v_{i}^{t}\|^{2}+\mathbb{E}\|v_{i}^{t} -\nabla f_{i}(z^{t-1})\|^{2}]\right.\nonumber\\
&\quad\left.+L^{2}\mathbb{E}\|z^t-z^{t-1}\|^{2}\right).
\end{align}

Substituting the bounds for (VII) and (VIII) back into (\ref{eq9}) and (\ref{eq8}), we explicitly verify that the drift error accumulates as
\begin{align}
&\mathbb{E}\left\|x_{i,k+1}^{t}-z^t\right\|^{2}\nonumber\\
&\leq 8(k+1)\alpha^2L^2\sum_{\ell=0}^{k}\mathbb{E}\left\|\left(x_{i,l}^t-z^t\right)\right\|^{2}+16K^{2}\alpha^{2} B_{h}^{2}\nonumber\\
&\quad+12\alpha^2K^2\left(\mathbb{E}\|c_{i}^{t}-v_{i}^{t}\|^{2}+\mathbb{E}\|v_{i}^{t} -\nabla f_{i}(z^{t-1})\|^{2} \right.\nonumber\\
&\quad\left.+L^{2}\mathbb{E}\|z^t-z^{t-1}\|^{2}\right) + 4K^{2}\alpha^{2}\mathbb{E}\|G_{\beta}(z^t)\|^{2}+8K\alpha^2\frac{\sigma^2}{b}.
\end{align}
To simplify the recursion analysis, let $D_k \triangleq \sum_{\ell=0}^{k}\mathbb{E}\|x_{i,\ell}^{t}-z^{t}\|^{2}$ represent the cumulative drift, and let $\Lambda^t$ denote the sum of all state-independent terms (the last three lines in (36)). The inequality can then be written compactly as:
\begin{equation}
D_{k} - D_{k-1} \leq 8(k+1)\alpha^2L^2 D_{k-1} + \Lambda^t.
\end{equation}
Under the step-size condition $\alpha \le \frac{1}{8KL}$ (implying $8(k+1)\alpha^2 L^2 \le \frac{1}{8K}$), we have:
\begin{equation}
D_{k} \leq \left(1 + \frac{1}{8K}\right) D_{k-1} + \Lambda^t.
\end{equation}
Unrolling this recurrence from $k=K-1$ down to $0$ yields a geometric series. Using the standard bound $\sum_{\ell=0}^{K-2} (1 + \frac{1}{8K})^{\ell} \leq 1.15 K$, we obtain the final bound:
\begin{equation}
D_{K-1} \leq 1.15 K \Lambda^t.
\end{equation}
Averaging over all clients, dividing by $K$ and substituting the explicit form of $\Lambda^t$ completes the proof.

\end{IEEEproof}

\section{Proof of Theorem 1}
\begin{IEEEproof}Adding $\frac{8\beta}{\eta} \times$ (\ref{eq12}) to (\ref{eq10}), we have  
\begin{align}\label{eq15}
&\mathbb{E}\left[ F\left( z^{t+1} \right) \right]+\frac{8\beta}{\eta} \mathbb{E}\left\|v^{t+1}-\nabla f\left(z^{t}\right)\right\|^{2} \nonumber\\
&\leq \mathbb{E}\left[ F\left( z^{t} \right)\right] + \left( L\beta - \frac{1}{2} \right) \mathbb{E}\|G_{\beta}(z^{t})\|^{2}  \nonumber\\
&\quad+ \left( \frac{L}{2} - \frac{1}{2\beta} \right) \mathbb{E}\left\| z^{t+1} - z^{t} \right\|^2 +\frac{8\beta}{\eta} \mathbb{E}\left\|v^{t}-\nabla f\left(z^{t-1}\right)\right\|^{2}\nonumber\\
&\quad+ \frac{6\beta}{N} \sum_{i} \left( \mathbb{E}  \left\| v_{i}^{t} - c_{i}^{t} \right\|^{2}  + \eta^{2} \mathbb{E}  \left\| v_{i}^{t} - \nabla f_{i}(z^{t-1}) \right\|^{2}  \right) \nonumber\\
&\quad + (10+\frac{16}{\eta^2})\beta L^{2} \mathbb{E} \left\| z^{t} - z^{t-1} \right\|^{2}+ (6\eta^{2}+16)\beta L^{2}U^{t} \nonumber\\
&\quad + \Big(\eta^2(2+\frac{2}{N})+\frac{16\eta}{N}\Big)\frac{\beta\sigma^{2}}{KB}.
\end{align}

Adding $\frac{7\beta}{1-q} \times$ (\ref{eq13}) to (\ref{eq15}), we have 
\begin{align}
&\mathbb{E}\left[ F\left( z^{t+1} \right) \right]+\frac{8\beta}{\eta} \mathbb{E}\left\|v^{t+1}-\nabla f\left(z^{t}\right)\right\|^{2} \nonumber\\
&\quad +\frac{7\beta}{1-q}\frac{1}{N}\sum_{i=1}^N\mathbb{E}  \left\| v_{i}^{t} - c_{i}^{t} \right\|^{2} \nonumber\\
&\leq \mathbb{E}\left[ F\left( z^{t} \right)\right] + \left( L\beta - \frac{1}{2} \right) \mathbb{E}\|G_{\beta}(z^{t})\|^{2}  \nonumber\\
&\quad+ \left( \frac{L}{2} - \frac{1}{2\beta} \right)\mathbb{E} \left\| z^{t+1} - z^{t} \right\|^2+\frac{8\beta}{\eta} \mathbb{E}\left\|v^{t}-\nabla f\left(z^{t-1}\right)\right\|^{2}\nonumber\\
&\quad+(\frac{7}{1-q}-1) \frac{\beta}{N} \sum_{i=1}^N  \mathbb{E} \left\| v_{i}^{t} - c_{i}^{t} \right\|^{2}  \nonumber\\
&\quad+ \frac{34\eta^2}{(1-q)^2}\frac{\beta}{N} \sum_{i=1}^N\mathbb{E}  \left\| v_{i}^{t} - \nabla f_{i}(z^{t-1}) \right\|^{2} \nonumber\\
& \quad+ (\frac{26}{\eta^2}+\frac{28\eta^2q^2}{(1-q)^2})\beta L^{2} \mathbb{E} \left\| z^{t} - z^{t-1} \right\|^{2}  \nonumber\\
&\quad+ (6\eta^{2}+16+\frac{21\eta^2q^2}{(1-q)^2})\beta L^{2}U^{t} + \Big(\eta^2(2+\frac{2}{N})\nonumber\\
&\quad+\frac{16\eta}{N}+\frac{14\eta^2q^2}{1-q}\Big)\frac{\beta\sigma^{2}}{KB}\nonumber.
\end{align}

Using $q,\eta\in[0,1]$ to simplify coefficients, we obtain that
\begin{align}\label{eq16}
&\mathbb{E}\left[ F\left( z^{t+1} \right) \right]+\frac{8\beta}{\eta} \mathbb{E}\left\|v^{t+1}-\nabla f\left(z^{t}\right)\right\|^{2} \nonumber\\
&\quad +\frac{7\beta}{1-q}\frac{1}{N}\sum_{i=1}^N\mathbb{E}  \left\| v_{i}^{t} - c_{i}^{t} \right\|^{2} \nonumber\\
&\leq \mathbb{E}\left[ F\left( z^{t} \right)\right]+ \left( \frac{L}{2} - \frac{1}{2\beta} \right) \mathbb{E}\left\| z^{t+1} - z^{t} \right\|^2   \nonumber\\
&\quad+ \left( L\beta - \frac{1}{2} \right) \mathbb{E}\|G_{\beta}(z^{t})\|^{2} +\frac{8\beta}{\eta} \mathbb{E}\left\|v^{t}-\nabla f\left(z^{t-1}\right)\right\|^{2}\nonumber\\
&\quad+(\frac{7}{1-q}-1) \frac{\beta}{N} \sum_{i=1}^N  \mathbb{E} \left\| v_{i}^{t} - c_{i}^{t} \right\|^{2}  \nonumber\\
&\quad+ \frac{34\eta^2}{(1-q)^2}\frac{\beta}{N} \sum_{i=1}^N\mathbb{E} \left\| v_{i}^{t} - \nabla f_{i}(z^{t-1}) \right\|^{2}   \nonumber\\
& \quad+ (\frac{26}{\eta^2}+\frac{28\eta^2q^2}{(1-q)^2})\beta L^{2} \mathbb{E}  \left\| z^{t} - z^{t-1} \right\|^{2} \nonumber\\
&\quad+ (16+\frac{21\eta^2}{(1-q)^2})\beta L^{2}U^{t} + \Big(\frac{16\eta}{N}+\frac{14\eta^2}{1-q}\Big)\frac{\beta\sigma^{2}}{KB}.
\end{align}

Adding $\frac{70\eta \beta}{(1-q)^2} \times$ (\ref{eq11}) to (\ref{eq16}) and using $q,\eta\in[0,1]$ to simplify coefficients, we have 
\begin{align}\label{eq17}
&\mathbb{E}\left[ F\left( z^{t+1} \right) \right]+\frac{8\beta}{\eta} \mathbb{E}\left\|v^{t+1}-\nabla f\left(z^{t}\right)\right\|^{2} \nonumber\\
&\quad +\frac{7\beta}{1-q}\frac{1}{N}\sum_{i=1}^N\mathbb{E}  \left\| v_{i}^{t} - c_{i}^{t} \right\|^{2}\nonumber\\
&\quad+\frac{70\eta \beta}{(1-q)^2}\frac{1}{N} \sum_{i=1}^N\mathbb{E} \left\| v_{i}^{t} - \nabla f_{i}(z^{t-1}) \right\|^{2} \nonumber\\
&\leq \mathbb{E}\left[ F\left( z^{t} \right)\right] + \left( L\beta - \frac{1}{2} \right)\mathbb{E} \|G_{\beta}(z^{t})\|^{2}  \nonumber\\
&\quad+ \left( \frac{L}{2} - \frac{1}{2\beta} \right)\mathbb{E} \left\| z^{t+1} - z^{t} \right\|^2 +\frac{8\beta}{\eta} \mathbb{E}\left\|v^{t}-\nabla f\left(z^{t-1}\right)\right\|^{2}\nonumber\\
&\quad+(\frac{7}{1-q}-1) \frac{\beta}{N} \sum_{i=1}^N  \mathbb{E}  \left\| v_{i}^{t} - c_{i}^{t} \right\|^{2} \nonumber\\
&\quad+(\frac{70\eta}{(1-q)^2}- \frac{\eta^2}{(1-q)^2})\frac{\beta}{N} \sum_{i=1}^N\mathbb{E}  \left\| v_{i}^{t} - \nabla f_{i}(z^{t-1}) \right\|^{2}   \nonumber\\
& \quad+ (\frac{26}{\eta^2}+\frac{28\eta^2q^2}{(1-q)^2}+\frac{140}{(1-q)^2})\beta L^{2} \mathbb{E}  \left\| z^{t} - z^{t-1} \right\|^{2}  \nonumber\\
&\quad+ (16+\frac{161\eta^2}{(1-q)^2})\beta L^{2}U^{t}\nonumber\\
&\quad+ \Big(\frac{16\eta}{N}+\frac{14\eta^2}{1-q}+\frac{140\eta^3}{(1-q)^2}\Big)\frac{\beta\sigma^{2}}{KB}.
\end{align}

Using $1.15\cdot16K^2\alpha^2L^2(16+\frac{161\eta^2}{(1-q)^2})\leq \eta^2$ and Lemma 2 gives
\begin{align}\label{eq18}
&\beta L^2(16+\frac{161\eta^2}{(1-q)^2}) U^t\nonumber\\
&\leq \frac{\beta\eta^2}{(1-q)^2} \frac{1}{N}\sum_{i=1}^{N}\mathbb{E}\|v_i^t-\nabla f_i(z^{t-1})\|^{2}+0.25\beta\mathbb{E}\|G_{\beta}(z^{t})\|^{2}\nonumber\\
&\quad+\beta \frac{1}{N}\sum_{i=1}^{N}\mathbb{E}\|v_i^t-c_i^t\|^{2}+\beta L^2\mathbb{E}\|z^{t}-z^{t-1}\|^{2}\nonumber\\
&\quad + 0.5\beta\frac{\sigma^2}{KB} + 18.4(16+\frac{161\eta^2}{(1-q)^2})\frac{L^2\beta^3}{\eta_g^2}B_h^2.
\end{align}

By defining the auxiliary function
\begin{align*}
&\Psi^t = \mathbb{E}\left[F(z^{t})\right] + \frac{70\eta \beta}{(1-q)^2} \frac{1}{N} \sum_i \mathbb{E}\left\|v_i^t - \nabla f_i(x^{t-1})\right\|^2 \\
&\quad+ \frac{8\beta }{\eta} \mathbb{E}\left\|v^t - \nabla f(x^{t-1})\right\|^2+ \frac{17\beta}{1-q} \frac{1}{N} \sum_i \mathbb{E}\left\|v_i^t - c_i^t\right\|^2,
\end{align*}
and combining (\ref{eq17}) and (\ref{eq18}), we can obtain that
\begin{align}\label{eq19}
&\Psi^{t+1}-\Psi^t\nonumber\\
&\leq
\beta L^2(\frac{27}{\eta^2}+\frac{168}{(1-q)^2})\mathbb{E}\|z^{t}-z^{t-1}\|^{2}-0.15\beta\|G_{\beta}(z^{t})\|^{2}\nonumber\\
&\quad+(\frac{L}{2}-\frac{1}{2\beta})\mathbb{E}\|z^{t+1}-z^{t}\|^{2}  + 18.4(16+\frac{161\eta^2}{(1-q)^2})\frac{L^2\beta^3}{\eta_g^2}B_h^2\nonumber\\
& \quad+ \beta(\frac{17\eta}{N}+\frac{14\eta^2}{1-q}+\frac{140\eta^3}{(1-q)^2})\frac{\sigma^2}{KB}.
\end{align}

Due to the step-size condition (\ref{eq:step_size_condition}), it holds that
$$\frac{1}{2\beta}\geq \frac{L}{2}+ \beta L^2(\frac{27}{\eta^2}+\frac{168}{(1-q)^2}),$$
Averaging (\ref{eq19}) over T and noting $\|z^0-z^{-1}\|^{2}=0$, we can obtain that
\begin{align}
&\frac{1}{T}\sum_{t=0}^{T-1} \mathbb{E}\|G_{\beta}(z^{t})\|^{2} \nonumber\\
&\leq \frac{\Psi^{0}}{0.15\beta T}+6.7 (\frac{17\eta}{N}+\frac{14\eta^2}{1-q}+\frac{140\eta^3}{(1-q)^2})\frac{\sigma^2}{KB}\nonumber\\
& \quad + 18.4(16+\frac{161\eta^2}{(1-q)^2})\frac{L^2\beta^2}{0.15\eta_g^2}B_h^2,
\end{align}
which completes the proof.
\end{IEEEproof}

\vspace{-1.1cm}
\begin{IEEEbiography}
	[{\includegraphics[width=1.1in,height=1.25in,clip,keepaspectratio]{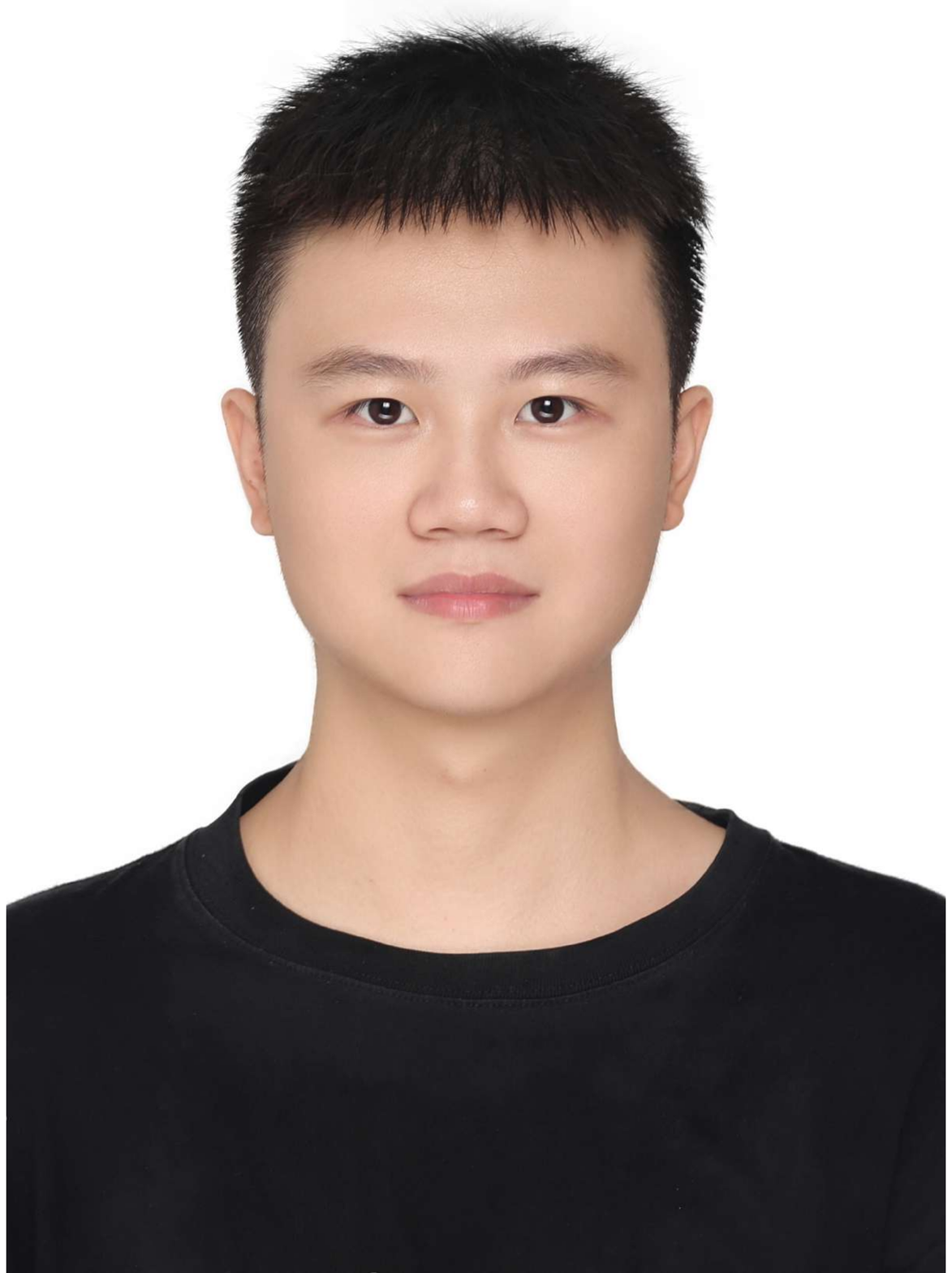}}] 
    {Pu Qiu}
	received the B.E. degree in intelligence science and technology from Sun Yat-sen University, Shenzhen, China, in 2024. He is now pursuing his M.E. degree in electronic information with the School of Intelligent Systems Engineering, Sun Yat-sen University. His current research interests include distributed optimization and machine learning.
\end{IEEEbiography}
\vspace{-1cm}
\begin{IEEEbiography}
	[{\includegraphics[width=1.1in,height=1.25in,clip,keepaspectratio]{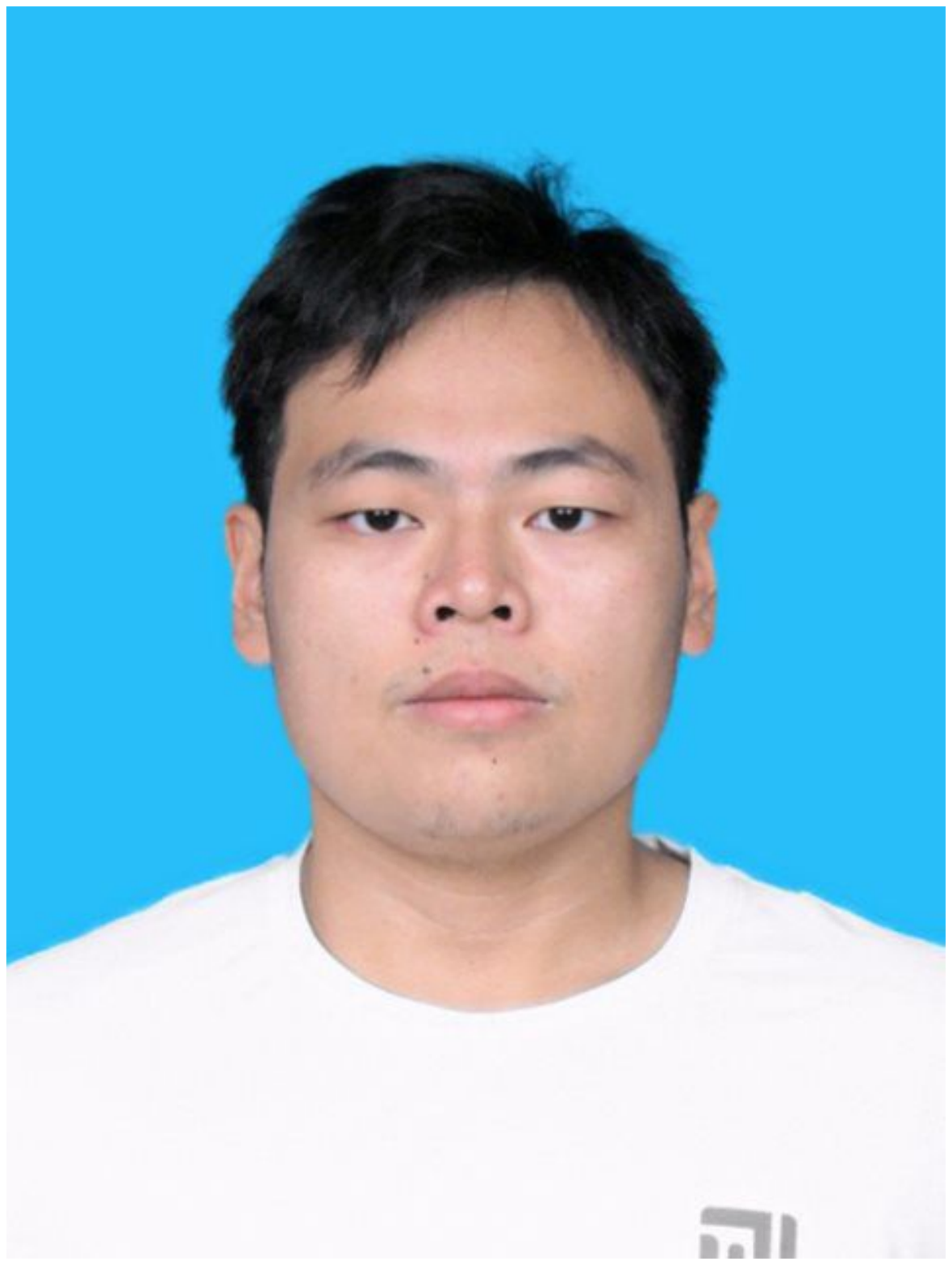}}]{Chen Ouyang}
	received the B.S. degree in information and computational science, the M.E. degrees in School of Mathematics and Information Science from Guangxi University, Nanning, China, in 2022, 2025, respectively. He is now pursuing his Ph.D. in the School of Intelligent Systems Engineering, Sun Yat-sen University. His current research interests include distributed optimization and machine learning.
\end{IEEEbiography}
\vspace{-0cm}
\begin{IEEEbiography}
	[{\includegraphics[width=1in,height=1.25in,clip,keepaspectratio]{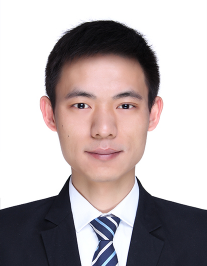}}]{Yongyang Xiong}
	received the B.S. degree in information and computational science, the M.E. and Ph.D. degrees in control science and engineering from Harbin Institute of Technology, Harbin, China, in 2012, 2014, and 2020, respectively. From 2017 to 2018, he was a joint Ph.D. student with the School of Electrical and Electronic Engineering, Nanyang Technological University, Singapore. From 2021 to 2024, he was a Postdoctoral Fellow with the Department of Automation, Tsinghua University, Beijing, China. Currently, he is an associate professor with the School of Intelligent Systems Engineering, Sun Yat-sen University, Shenzhen, China. His current research interests include networked control systems, distributed optimization and learning, multi-agent reinforcement learning and their applications.	
\end{IEEEbiography}
\vspace{-3cm}
\begin{IEEEbiography}
	[{\includegraphics[width=1in,height=1.25in,clip,keepaspectratio]{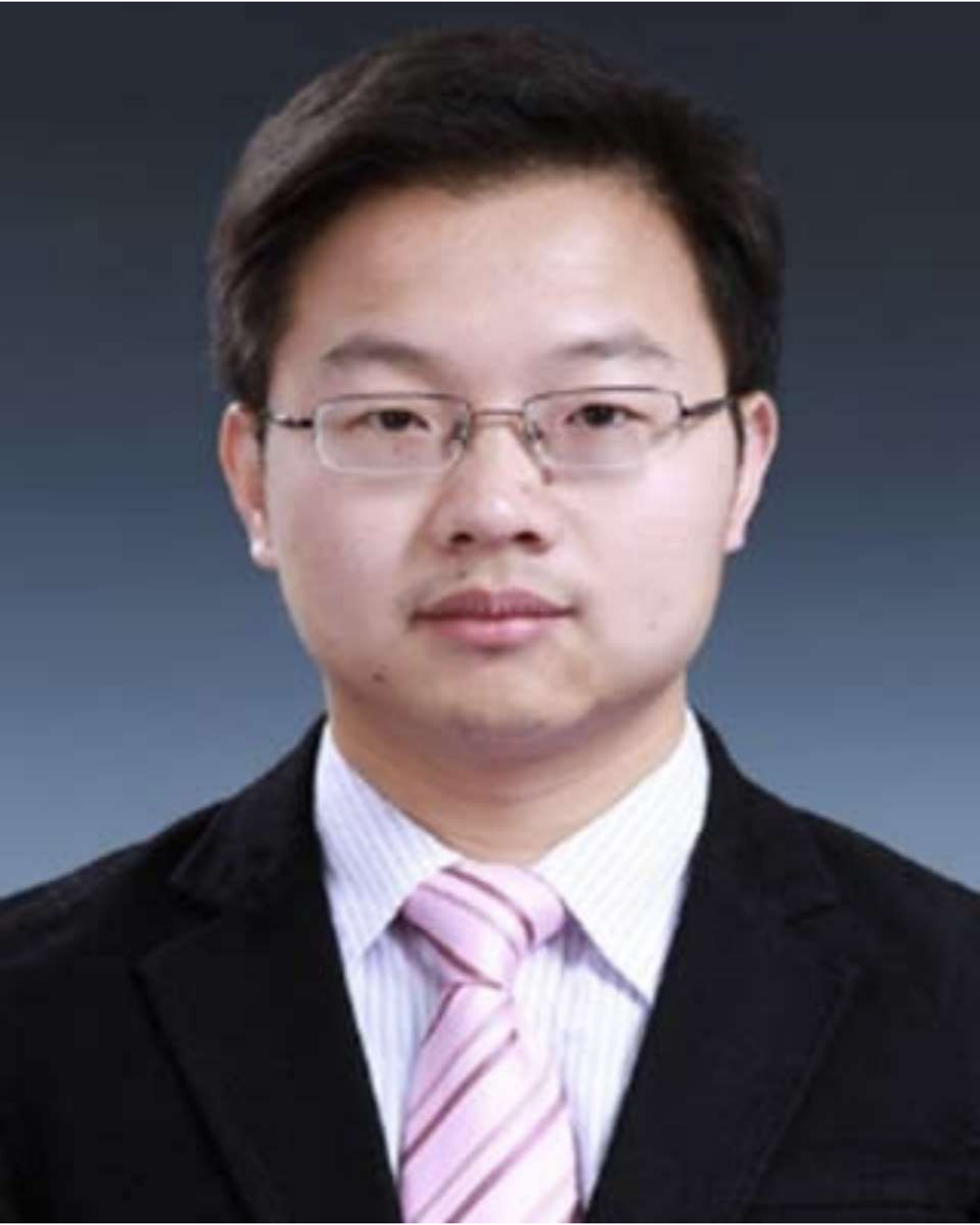}}]{Keyou You}
	(Senior Member, IEEE) received the B.S. degree in statistical science from Sun Yat-sen University, Guangzhou, China, in 2007 and the Ph.D. degree in electrical and electronic engineering from Nanyang Technological University (NTU), Singapore, in 2012. 
	
	After briefly working as a Research Fellow at NTU, he joined Tsinghua University, Beijing, China, where he is currently a Full Professor in the Department of Automation. He held visiting positions with Politecnico di Torino, Turin, Italy, Hong Kong University of Science and Technology, Hong Kong, China, University of Melbourne, Melbourne, Victoria, Australia, and so on. His research interests include the intersections between control, optimization and learning, as well as their applications in autonomous systems. 
	
	Dr. You received the Guan Zhaozhi Award at the 29th Chinese Control Conference in 2010 and the ACA (Asian Control Association) Temasek Young Educator Award in 2019. He received the National Science Funds for Excellent Young Scholars in 2017 and for Distinguished Young Scholars in 2023. He is currently an Associate Editor for \textit{Automatica} and IEEE TRANSACTIONS ON CONTROL OF NETWORK SYSTEMS.
\end{IEEEbiography}
\vspace{-3cm}
\begin{IEEEbiography}
	[{\includegraphics[width=1in,height=1.25in,clip,keepaspectratio]{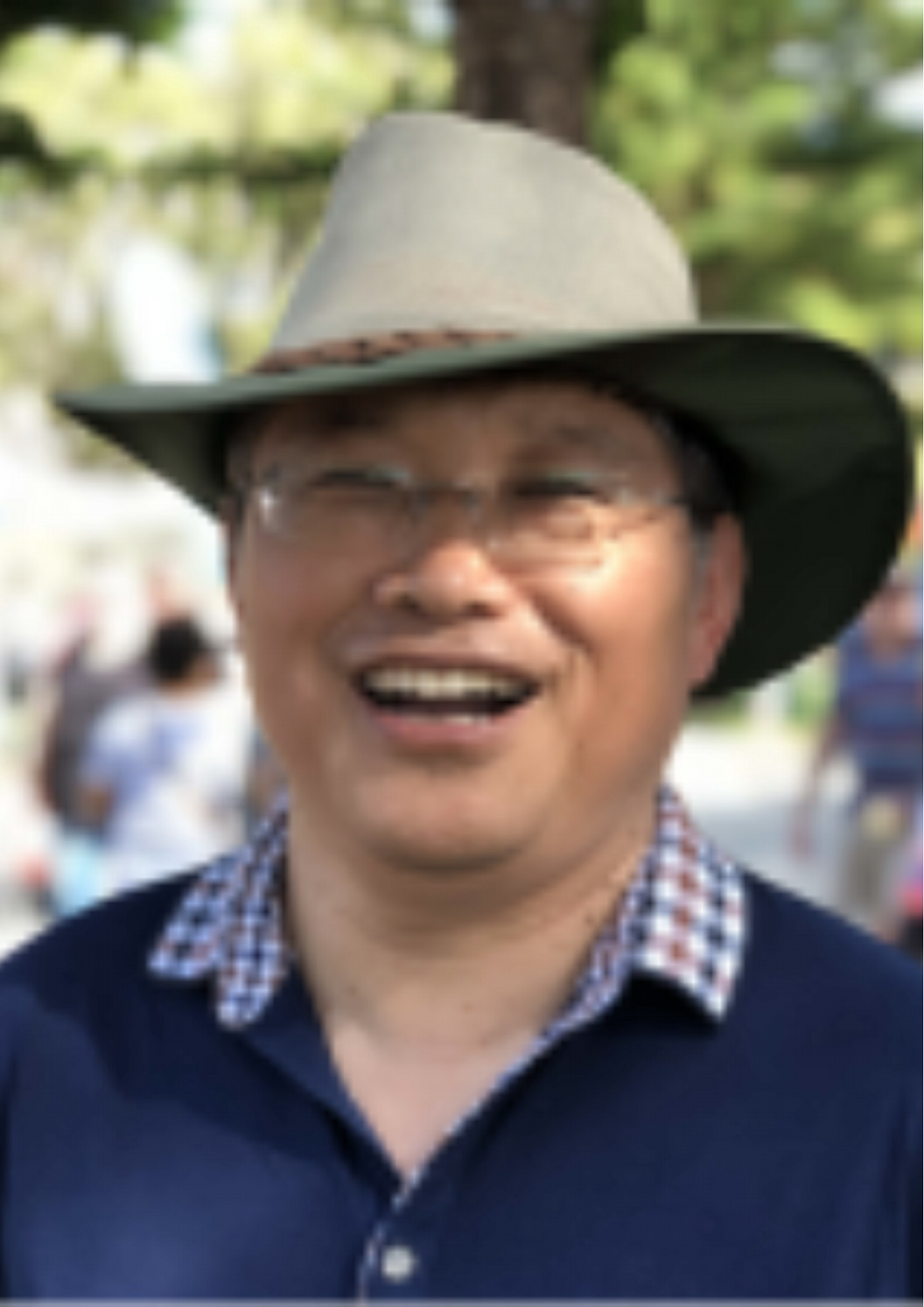}}]{Wanquan Liu}
(Senior Member, IEEE) received the B.S. degree in Applied Mathematics from Qufu Normal University, China, in 1985, the M.S. degree in Control
Theory and Operation Research from Chinese
Academy of Science in 1988, and the Ph.D.
degree in Electrical Engineering from Shanghai
Jiaotong University, in 1993. He once held the
ARC Fellowship, U2000 Fellowship and JSPS
Fellowship and attracted research funds from
different resources over 2.4 million dollars. He is
currently a Full Professor at the School of Intelligent Systems Engineering, Sun Yat-Sen University, Shenzhen, China.
His current research interests include large-scale pattern recognition, signal processing, machine learning, and control systems.
\end{IEEEbiography}

\newpage
\begin{IEEEbiography}
	[{\includegraphics[width=1in,height=1.25in,clip,keepaspectratio]{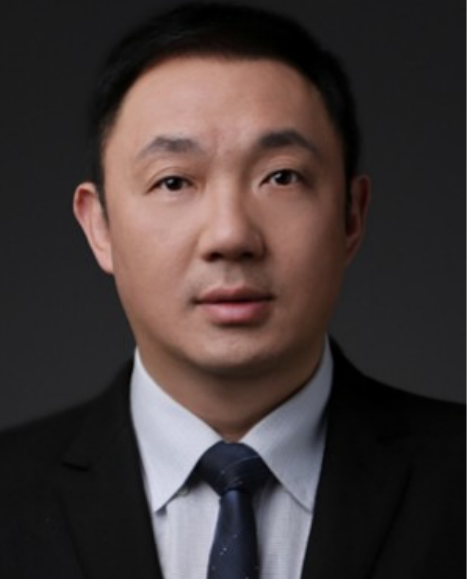}}]{Yang Shi} (Fellow, IEEE) 
	received the B.Sc. and Ph.D. degrees in mechanical engineering and
	automatic control from Northwestern Polytechnical University, Xi’an, China, in 1994 and 1998, respectively, and the Ph.D. degree in electrical and computer engineering from the University of Alberta, Edmonton, AB, Canada, in 2005. He was a Research Associate with the Department of Automation, Tsinghua University, China, from 1998 to 2000. From 2005 to 2009, he was an Assistant Professor and an Associate Professor with the Department of Mechanical Engineering, University of Saskatchewan, Saskatoon, SK, Canada. In 2009, he joined the University of Victoria, and currentlu he is a Professor with the Department of Mechanical Engineering, University of Victoria, Victoria, BC, Canada. His current research interests include networked and distributed systems, model predictive control (MPC), cyber-physical systems (CPS), robotics and mechatronics, navigation and control of autonomous systems (AUV and UAV), and energy system applications.
	
	Dr. Shi is the IFAC Council Member. He is a fellow of ASME, CSME,
	Engineering Institute of Canada (EIC), Canadian Academy of Engineering (CAE), Royal Society of Canada (RSC), and a registered Professional Engineer in British Columbia and Canada. He received the University of Saskatchewan Student Union Teaching Excellence Award in 2007, the Faculty of Engineering Teaching Excellence Award in 2012 at the University of Victoria (UVic), and the 2023 REACH Award for Excellence in Graduate Student Supervision and Mentorship. On research, he was a recipient of the JSPS Invitation Fellowship (short-term) in 2013, the UVic Craigdarroch Silver Medal for Excellence in Research in 2015, the Humboldt Research Fellowship for Experienced Researchers in 2018, CSME Mechatronics Medal in 2023, the IEEE Dr.-Ing. Eugene Mittelmann Achievement Award in 2023,
	the 2024 IEEE Canada Outstanding Engineer Award. He was a Vice-President on Conference Activities of IEEE IES from 2022 to 2025 and the Chair of IEEE IES Technical Committee on Industrial Cyber-Physical Systems. Currently, he is the Editor-in-Chief of IEEE TRANSACTIONS ON INDUSTRIAL ELECTRONICS. He also serves as Associate Editor for \textit{Automatica}, IEEE TRANSACTIONS ON AUTOMATIC CONTROL, and \textit{Annual Review in Controls}.
\end{IEEEbiography}


 




\vfill

\end{document}